\documentclass{amsart}
\usepackage[dvips,final]{graphics}
\usepackage{array}
\usepackage{arydshln}
\usepackage[makeroom]{cancel}
 \usepackage[all]{xy}
 \usepackage{url}
\usepackage{multirow, blkarray}
\usepackage{booktabs}
\usepackage{textcomp}
 \usepackage[final]{epsfig}
 \usepackage{color}
\usepackage[T1]{fontenc}      
\usepackage[english,french]{babel}
\usepackage[utf8]{inputenc}
\usepackage{blindtext}

\usepackage{amsfonts,amscd,array, mathdots, epigraph}
\usepackage{amsmath}
\usepackage{amssymb}
\usepackage{amsthm}
\usepackage{mathrsfs}
\usepackage{stmaryrd}
\usepackage{slashbox}
\usepackage{diagbox}
\usepackage{enumitem}

\usepackage{ulem}
\usepackage{tikz}
\usepackage{xcolor}
\usepackage{multicol}
\definecolor{ufogreen}{rgb}{0.24, 0.82, 0.44}

\begin{document}


\newtheorem{theorem}{Théorème}[section]
\newtheorem{theore}{Théorème}
\newtheorem{definition}[theorem]{Définition}
\newtheorem{proposition}[theorem]{Proposition}
\newtheorem{corollary}[theorem]{Corollaire}
\newtheorem*{con}{Conjecture}
\newtheorem*{remark}{Remarque}
\newtheorem*{remarks}{Remarques}
\newtheorem*{pro}{Problème}
\newtheorem*{examples}{Exemples}
\newtheorem*{example}{Exemple}
\newtheorem{lemma}[theorem]{Lemme}


\title{Classification des entiers monomialement irréductibles et généralisations}

\author{Flavien Mabilat}

\date{}


\address{
Flavien Mabilat,
Laboratoire de Mathématiques de Reims,
UMR9008 CNRS et Université de Reims Champagne-Ardenne, 
U.F.R. Sciences Exactes et Naturelles 
Moulin de la Housse - BP 1039 
51687 Reims cedex 2,
France
}
\def\emailaddrname{{\itshape Courriel}}
\email{flavien.mabilat@univ-reims.fr}

\maketitle

\selectlanguage{french}
\begin{abstract}
Dans cet article, on va s'intéresser à la classification de certains entiers naturels reliés à la combinatoire des sous-groupes de congruence du groupe modulaire. Plus précisément, on va s'intéresser ici à la notion de solutions monomiales minimales. Celles-ci sont les solutions d'une équation matricielle (apparaissant également lors de l'étude des frises de Coxeter), modulo un entier $N$, dont toutes les composantes sont identiques et minimales pour cette propriété. Notre objectif ici est d'étudier les entiers $N$ pour lesquels les solutions monomiales minimales vérifiant certaines conditions fixées possèdent une propriété d'irréductibilité. En particulier, on effectuera la classification des entiers monomialement irréductibles qui sont les entiers pour lesquels toutes les solutions monomiales minimales non nulles sont irréductibles.
\\
\end{abstract}

\selectlanguage{english}
\begin{abstract}
In this article, we study the classification of some natural numbers related to the combinatorics of congruence subgroups of the modular group. More precisely, we will focus here on the notion of minimal monomial solutions. These are the solutions of a matrix equation (also appearing in the study of Coxeter's friezes), modulo an integer $N$, whose components are identical and minimal for this property. Our aim here is to study the integers $N$ for which the minimal monomial solutions satisfying some fixed conditions have an irreducibility property. In particular, we will classify the monomially irreducible integers which are the integers for which all the nonzero minimal monomial solutions are irreducible.

\end{abstract}

\selectlanguage{french}

\thispagestyle{empty}

\noindent {\bf Mots clés:} groupe modulaire; solution monomiale; irréductibilité; classes modulo $N$  
\\
\begin{flushright}
\og \textit{Le langage exprime autant par ce qui est entre les mots que par les mots eux-mêmes,
\\et par ce qu'il ne dit pas que par ce qu'il dit} \fg
\\Maurice Merleau-Ponty, \textit{La Prose du monde.}
\end{flushright}

\section{Introduction}

Depuis les résultats obtenus par J.\ L.\ Lagrange et C.\ F.\ Gauss sur les formes quadratiques à deux variables, le groupe modulaire \[SL_{2}(\mathbb{Z})=
\left\{
\begin{pmatrix}
a & b \\
c & d
   \end{pmatrix}
 \;\vert\;a,b,c,d \in \mathbb{Z},\;
 ad-bc=1
\right\}\] et son quotient $PSL_{2}(\mathbb{Z})=SL_{2}(\mathbb{Z})/\{\pm Id\}$ ont été au centre de très nombreux travaux. Grâce à ces derniers, on dispose d'une connaissance étendue de la structure de ces deux groupes et des propriétés vérifiées par leurs éléments. En particulier, un des résultats les plus anciens et les plus importants concernant $SL_{2}(\mathbb{Z})$ est l'existence de parties génératrices à seulement deux éléments. On peut notamment, parmi tous les choix possibles, considérer les deux matrices suivantes :
\[T=\begin{pmatrix}
 1 & 1 \\[2pt]
    0    & 1 
   \end{pmatrix}, S=\begin{pmatrix}
   0 & -1 \\[2pt]
    1    & 0 
   \end{pmatrix}.
 \] 

\noindent Une fois ces deux générateurs sélectionnés, on peut montrer (voir par exemple l'introduction de \cite{M}) que, pour tout élément $A$ du groupe modulaire, il existe un entier strictement positif $n$ et des entiers strictement positifs $a_{1},\ldots,a_{n}$ tels que : \[A=T^{a_{n}}ST^{a_{n-1}}S\cdots T^{a_{1}}S=\begin{pmatrix}
   a_{n} & -1 \\[4pt]
    1    & 0 
   \end{pmatrix}
\begin{pmatrix}
   a_{n-1} & -1 \\[4pt]
    1    & 0 
   \end{pmatrix}
   \cdots
   \begin{pmatrix}
   a_{1} & -1 \\[4pt]
    1    & 0 
    \end{pmatrix}:=M_{n}(a_{1},\ldots,a_{n}).\]
		
\noindent En revanche, une telle écriture n'est pas unique. En effet, on dispose, par exemple, de l'égalité suivante : \[-Id=M_{3}(1,1,1)=M_{4}(1,2,1,2).\]

Ainsi, tout élément du groupe modulaire peut s'écrire sous la forme $M_{n}(a_{1},\ldots,a_{n})$ avec $a_{1},\ldots,a_{n}$ des entiers strictement positifs. La possibilité d'exprimer les matrices de $SL_{2}(\mathbb{Z})$ sous cette forme est d'autant plus intéressante que les matrices $M_{n}(a_{1},\ldots,a_{n})$ interviennent dans de nombreux autres domaines. On les retrouve notamment dans l'étude des réduites des fractions continues de Hirzebruch-Jung, dans la construction des frises de Coxeter ou encore dans l'écriture matricielle des équations de Sturm-Liouville discrètes (voir par exemple \cite{O} section 1.3). 
\\
\\ \indent Bien qu'elle soit connue depuis longtemps, cette façon d'exprimer les éléments du groupe modulaire a  été jusqu'à présent assez peu exploitée. Cependant, de nouveaux horizons de recherche liés à cette dernière ont récemment émergé. Parmi ceux-ci, il y a notamment les différentes écritures associées aux matrices $\pm Id$, c'est-à-dire la résolution sur $\mathbb{N}^{*}$ de l'équation suivante (parfois appelée équation de Conway-Coxeter) : \begin{equation}
\label{a}
\tag{$E$}
M_{n}(a_1,\ldots,a_n)=\pm Id.
\end{equation} 

\noindent V.\ Ovsienko (voir \cite{O} Théorèmes 1 et 2) a entièrement résolu cette dernière sur les entiers naturels non nuls et a donné une description combinatoire des solutions en terme de découpages de polygones. Ce résultat permet d'envisager plusieurs pistes de généralisations. On peut ainsi, en lien avec la construction des frises de Coxeter, résoudre \eqref{a} sur d'autres ensembles (voir \cite{C,CH}). On peut également chercher les différentes écritures associées à d'autres éléments du groupe modulaire, notamment $S$ et $T$ (voir \cite{M}). Ici, on va s'intéresser à un autre angle d'étude. En effet, le résultat de V.\ Ovsienko peut être vu comme une caractérisation des $n$-uplets d'entiers strictement positifs pour lesquels $M_{n}(a_{1},\ldots,a_{n})$ appartient au centre de $SL_{2}(\mathbb{Z})$. À la lueur de cette interprétation, on est naturellement amené à chercher toutes les écritures des éléments appartenant à certains sous-groupes fixés. Dans ce qui va suivre, on va considérer les cas des sous-groupes de congruence suivants: \[\{A \in SL_{2}(\mathbb{Z})~{\rm tel~que}~A= \pm Id~( {\rm mod}~N)\},\]
\noindent ce qui est en fait équivalent à l'étude sur $\mathbb{Z}/N\mathbb{Z}$ de l'équation :
\begin{equation}
\label{p}
\tag{$E_{N}$}
M_{n}(a_1,\ldots,a_n)=\pm Id.
\end{equation} 
\noindent On dira, en particulier, qu'une solution de \eqref{p} est de taille $n$ si cette solution est un $n$-uplet d'éléments de $\mathbb{Z}/N\mathbb{Z}$.
\\
\\ \indent On dispose déjà de nombreux résultats concernant l'équation \eqref{p} (voir notamment \cite{M1, M2, M3, M4}). L'élément clef utilisé pour l'obtention de ces derniers est le recours à une notion de solution irréductible, introduite à l'origine pour l'étude des frises de Coxeter, à partir de laquelle on peut construire l'ensemble des solutions (voir section suivante). Grâce à elle, on a pu résoudre complètement \eqref{p} pour $N \leq 6$ (voir \cite{M1} section 4) et on a pu établir plusieurs résultats généraux d'irréductibilité. 
\\
\\ \indent Une part importante de ces résultats concernent les solutions monomiales minimales qui sont les solutions de \eqref{p} dont toutes les composantes sont identiques et minimales pour cette propriété (voir \cite{M1} section 3.3 et la section suivante). De plus, on sait que toutes les solutions monomiales minimales non nulles sont irréductibles pour certains entiers $N$ (voir \cite{M1}). Ces deux éléments incitent naturellement à rechercher l'ensemble des entiers monomialement irréductibles, qui sont les entiers pour lesquels toutes les solutions monomiales minimales non nulles sont irréductibles. Dans cette optique, on avait obtenu, dans un précédent article, un début de classification de ces derniers (voir \cite{M3}) et on cherche ici à obtenir une classification complète. Une fois celle-ci établie, on s'intéressera à des généralisations du concept d'entier monomialement irréductible. Pour mener à bien ces différentes tâches, on va donner, dans la partie suivante, les définitions qui nous seront utiles pour la suite ainsi que les résultats principaux qui seront démontrés dans les parties suivantes. 

\section{Définitions et résultats principaux}
\label{RP}    

Cette section poursuit un double objectif. D'une part, elle vise à fournir les définitions et notations essentielles à l'étude de l'équation \eqref{p} qui seront fréquemment utilisées dans les sections suivantes. D'autre part, elle va permettre de regrouper les énoncés des résultats les plus importants contenus dans ce texte. Sauf mention contraire, $N$ désigne un entier naturel supérieur à $2$. Si $A \in SL_{2}(\mathbb{Z}/N\mathbb{Z})$ on note $A^{T}$ la transposée de $A$. On aura régulièrement besoin de considérer des classes d'entiers modulo $N$ et modulo des diviseurs de $N$. Aussi, s'il n'y a pas d'ambiguïté sur $N$, on utilisera la notation $\overline{a}:=a+N\mathbb{Z}$ (avec $a \in \mathbb{Z}$) pour désigner les classes modulo $N$ et on écrira $a+\frac{N}{k}\mathbb{Z}$ pour représenter les classes modulo un diviseur de $N$. $\mathbb{P}$ désigne l'ensemble des nombres premiers et $\varphi$ la fonction indicatrice d'Euler.

\begin{definition}[\cite{C}, lemme 2.7]
\label{21}

Soient $(\overline{a_{1}},\ldots,\overline{a_{n}}) \in (\mathbb{Z}/N \mathbb{Z})^{n}$ et $(\overline{b_{1}},\ldots,\overline{b_{m}}) \in (\mathbb{Z}/N \mathbb{Z})^{m}$. On définit l'opération ci-dessous: \[(\overline{a_{1}},\ldots,\overline{a_{n}}) \oplus (\overline{b_{1}},\ldots,\overline{b_{m}})= (\overline{a_{1}+b_{m}},\overline{a_{2}},\ldots,\overline{a_{n-1}},\overline{a_{n}+b_{1}},\overline{b_{2}},\ldots,\overline{b_{m-1}}).\] Le $(n+m-2)$-uplet obtenu est appelé la somme de $(\overline{a_{1}},\ldots,\overline{a_{n}})$ avec $(\overline{b_{1}},\ldots,\overline{b_{m}})$.

\end{definition}

\begin{examples}

{\rm On donne ci-dessous quelques exemples de sommes :
\begin{itemize}
\item $(\overline{2},\overline{0},\overline{5}) \oplus (\overline{-1},\overline{2},\overline{1})=(\overline{3},\overline{0},\overline{4},\overline{2})$;
\item $(\overline{3},\overline{1},\overline{2},\overline{0}) \oplus (\overline{2},\overline{2},\overline{1},\overline{5},\overline{1})=(\overline{4},\overline{1},\overline{2},\overline{2},\overline{2},\overline{1},\overline{5})$;
\item $n \geq 2$, $(\overline{a_{1}},\ldots,\overline{a_{n}}) \oplus (\overline{0},\overline{0}) = (\overline{0},\overline{0}) \oplus (\overline{a_{1}},\ldots,\overline{a_{n}})=(\overline{a_{1}},\ldots,\overline{a_{n}})$.
\end{itemize}
}
\end{examples}

L'opération $\oplus$ introduite dans la définition précédente est particulièrement utile pour l'étude de l'équation \eqref{p}. En effet, celle-ci possède la propriété suivante : si $(\overline{b_{1}},\ldots,\overline{b_{m}})$ est une solution de \eqref{p} alors la somme $(\overline{a_{1}},\ldots,\overline{a_{n}}) \oplus (\overline{b_{1}},\ldots,\overline{b_{m}})$ est une solution de \eqref{p} si et seulement si $(\overline{a_{1}},\ldots,\overline{a_{n}})$ est une solution de \eqref{p} (voir \cite{C,WZ} et \cite{M1} proposition 3.7). En revanche, $\oplus$ n'est ni commutative ni associative (voir \cite{WZ} exemple 2.1) et les $k$-uplets d'éléments de $\mathbb{Z}/N \mathbb{Z}$ n'ont pas d'inverse pour $\oplus$ lorsque $k \geq 3$.

\begin{definition}[\cite{C}, définition 2.5]
\label{22}

Soient $(\overline{a_{1}},\ldots,\overline{a_{n}}) \in (\mathbb{Z}/N \mathbb{Z})^{n}$ et $(\overline{b_{1}},\ldots,\overline{b_{n}}) \in (\mathbb{Z}/N \mathbb{Z})^{n}$. On dit que $(\overline{a_{1}},\ldots,\overline{a_{n}}) \sim (\overline{b_{1}},\ldots,\overline{b_{n}})$ si $(\overline{b_{1}},\ldots,\overline{b_{n}})$ est obtenu par permutations circulaires de $(\overline{a_{1}},\ldots,\overline{a_{n}})$ ou de $(\overline{a_{n}},\ldots,\overline{a_{1}})$.

\end{definition}

On peut aisément s'assurer que $\sim$ est une relation d'équivalence sur les $n$-uplets d'éléments de $\mathbb{Z}/N \mathbb{Z}$ (voir \cite{WZ} lemme 1.7). Par ailleurs, si $(\overline{a_{1}},\ldots,\overline{a_{n}}) \sim (\overline{b_{1}},\ldots,\overline{b_{n}})$ alors $(\overline{a_{1}},\ldots,\overline{a_{n}})$ est solution de \eqref{p} si et seulement si $(\overline{b_{1}},\ldots,\overline{b_{n}})$ l'est aussi (voir \cite{C} proposition 2.6).
\\
\\Muni de ces deux définitions, on peut maintenant définir la notion d'irréductibilité annoncée.

\begin{definition}[\cite{C}, définition 2.9]
\label{23}

Une solution $(\overline{c_{1}},\ldots,\overline{c_{n}})$ avec $n \geq 3$ de \eqref{p} est dite réductible s'il existe une solution de \eqref{p} $(\overline{b_{1}},\ldots,\overline{b_{l}})$ et un $m$-uplet $(\overline{a_{1}},\ldots,\overline{a_{m}})$ d'éléments de $\mathbb{Z}/N \mathbb{Z}$ tels que \begin{itemize}
\item $(\overline{c_{1}},\ldots,\overline{c_{n}}) \sim (\overline{a_{1}},\ldots,\overline{a_{m}}) \oplus (\overline{b_{1}},\ldots,\overline{b_{l}})$;
\item $m \geq 3$ et $l \geq 3$.
\end{itemize}
Une solution est dite irréductible si elle n'est pas réductible.

\end{definition}

\begin{remark} 

{\rm On ne considère pas $(\overline{0},\overline{0})$ comme une solution irréductible de \eqref{p}.}

\end{remark}

\indent Ainsi, à la lueur de cette définition, la résolution de \eqref{p} tend à se réduire à la recherche et à l'étude des solutions irréductibles. Cela a amené naturellement à l'introduction et à l'étude de classes particulières de solutions dont celle des solutions monomiales définie ci-dessous :

\begin{definition}[\cite{M1}, définition 3.9]
\label{24}

i)~Soit $\overline{k} \in \mathbb{Z}/N\mathbb{Z}$. On appelle solution $\overline{k}$-monomiale un $n$-uplet d'éléments de $\mathbb{Z}/ N \mathbb{Z}$ constitué uniquement de $\overline{k}$ et solution de \eqref{p}.
\\
\\ ii)~On appelle solution monomiale une solution pour laquelle il existe $\overline{l} \in \mathbb{Z}/N\mathbb{Z}$ tel qu'elle est $\overline{l}$-monomiale.
\\
\\ iii)~On appelle solution $\overline{k}$-monomiale minimale une solution $\overline{k}$-monomiale de taille $n$ avec $n$ le plus petit entier pour lequel il existe une solution $\overline{k}$-monomiale.
\\
\\ iv)~On appelle solution monomiale minimale une solution $\overline{k}$-monomiale minimale pour un $\overline{k} \in \mathbb{Z}/N\mathbb{Z}$.

\end{definition}

Notons que pour tout $\overline{k} \in \mathbb{Z}/N\mathbb{Z}$ la solution $\overline{k}$-monomiale minimale de \eqref{p} existe toujours puisque $M_{1}(\overline{k})$ est d'ordre fini dans ${\rm PSL}_{2}(\mathbb{Z}/N\mathbb{Z})$. De plus, le $n$-uplet $(\overline{k},\ldots,\overline{k})$ est une solution de \eqref{p} si et seulement si $n$ est un multiple de l'ordre de $M_{1}(\overline{k})$ dans ${\rm PSL}_{2}(\mathbb{Z}/N\mathbb{Z})$.
\\
\\ \indent On a déjà démontré de nombreuses propriétés d'irréductibilité pour ces solutions (voir \cite{M1, M2, M4} et la section \ref{pre}). Ces résultats montrent notamment que, pour certaines valeurs de $N$, toutes les solutions $\overline{k}$-monomiales minimales avec $k$ et $N$ premiers entre eux, voire même toutes les solutions monomiales minimales non nulles, sont irréductibles. Cela amène naturellement aux définitions ci-dessous : 
		
\begin{definition}
\label{25}

Soit $N$ un entier naturel non nul supérieur à 2.
\\
\\i) On dit que $N$ est monomialement irréductible si pour tout $\overline{k} \in \mathbb{Z}/N\mathbb{Z}-\{\overline{0}\}$ la solution $\overline{k}$-monomiale minimale de \eqref{p} est irréductible. Dans le cas contraire, on dit que $N$ est monomialement réductible.
\\
\\ii) $N$ est dit quasi monomialement irréductible si, pour tout $\overline{k} \in \mathbb{Z}/N\mathbb{Z}$ tel que $k$ et $N$ sont premiers entre eux, la solution $\overline{k}$-monomiale minimale de \eqref{p} est irréductible. Dans le cas contraire, on dit que $N$ est quasi monomialement réductible.

\end{definition}

\begin{remark}
{\rm En particulier, les entiers monomialement irréductibles sont quasi monomialement irréductibles. Notons également qu'il est tout à fait possible, si $N$ est quasi monomialement irréductible, qu'il existe un entier $l$ non premier avec $N$ tel que la solution $\overline{l}$-monomiale minimale de \eqref{p} est irréductible. Par exemple, si $N=16$ alors $N$ est quasi monomialement irréductible (mais pas monomialement irréductible) et la solution $\overline{6}$-monomiale minimale de \eqref{p} est irréductible (voir Théorème \ref{33}).

}
\end{remark}

\indent On souhaite procéder à la classification complète de ces entiers. Notre premier objectif est de terminer la classification des entiers monomialement irréductibles commencée dans \cite{M3}. Dans cette article, on avait obtenu un certain nombre de résultats, rappelés dans la section suivante (voir Théorème \ref{35}), et on avait émis la conjecture suivante : les entiers monomialement irréductibles sont premiers ou égaux à 4, 6, 8, 12 ou 24. On va montrer ici que cette conjecture est vraie, c'est-à-dire démontrer le résultat suivant : 

\begin{theorem}
\label{26}

Soit $N$ un entier supérieur à 2. $N$ est monomialement irréductible si et seulement si $N \in \mathbb{P} \cup \{4, 6, 8, 12, 24\}$.

\end{theorem}

\noindent Ensuite, on effectuera la classification complète des entiers quasi monomialement irréductibles.

\begin{theorem}
\label{27}

Soit $N$ un entier supérieur à 2. $N$ est quasi monomialement irréductible si et seulement si \[N \in \{p^{n}, p \in \mathbb{P}, n \in \mathbb{N}^{*}\} \cup \{2^{n}3^{m}, (n,m) \in (\mathbb{N}^{*})^{2}\}.\]

\end{theorem}

\noindent Ces deux théorèmes seront démontrés dans la section \ref{quasi}.
\\
\\ \indent On constate en particulier dans la liste précédente l'apparition des entiers de la forme $2^{n}3^{m}$ qui jusque-là n'avaient pas encore été considérés. Comme on dispose de la classification complète des solutions monomiales minimales irréductibles dans les cas où $N$ est la puissance d'un nombre premier (voir \cite{M4} et le Théorème \ref{33}), il semble intéressant d'obtenir un résultat semblable pour ce nouvel ensemble d'entiers. Dans cette optique, on établira dans la section \ref{class} le résultat partiel ci-dessous :

\begin{theorem}
\label{28}

Soit $N=2 \times 3^{m}$ avec $m \geq 2$. Soit $0 \leq k \leq N-1$. La solution $\overline{k}$-monomiale minimale de \eqref{p} est réductible si et seulement si $k=3^{l}a$ avec $1 \leq l \leq m-2$ ou $k=3^{m-1}a$ avec $a$ pair premier avec 3, c'est-à-dire si et seulement si $k$ est un multiple de 3 différent de $3^{m}$, de $3^{m-1}$ et de $5 \times 3^{m-1}$. En particulier, \eqref{p} a $4 \times 3^{m-1}+3$ solutions monomiales minimales irréductibles.

\end{theorem}

Lors de la classification des solutions monomiales minimales irréductibles dans les cas où $N$ est la puissance d'un nombre premier, on a vu apparaître les solutions $\overline{2a}$-monomiales minimales avec $a$ premier avec $N$. Cela amène à introduire le concept suivant :

\begin{definition}
\label{29}

Soit $N \geq 2$.
 
\begin{itemize}
\item Si $N$ est impair ou divisible par 4. $N$ est dit semi monomialement irréductible si, pour tout entier $a$ premier avec $N$, la solution $\overline{2a}$-monomiales minimales de \eqref{p} est irréductible.
\item Si $N$ est pair non divisible par 4. $N$ est dit semi monomialement irréductible si, pour tout entier $a$ premier avec $\frac{N}{2}$, la solution $\overline{2a}$-monomiale minimale de \eqref{p} est irréductible.
\end{itemize}

\noindent Dans le cas contraire, on dit que $N$ est semi monomialement réductible.

\end{definition}

\begin{remarks}
{\rm i) On doit scinder dans la définition précédente le cas $N$ pair non divisible par 4 des autres. En effet, si $N$ est pair non divisible par 4 et si $a$ est premier avec $N$ alors $\overline{2a}=\overline{2\left(a+\frac{N}{2}\right)}$ et $a+\frac{N}{2}$ est pair, donc non premier avec $N$. Une fois cette distinction effectuée, la définition devient cohérente puisque le caractère semi monomialement irréductible ne dépend plus des représentants choisis pour étudier l'irréductibilité des solutions.
\\
\\ii) 2 n'est pas semi monomialement réductible car si $a=1$ alors $\overline{2a}=\overline{0}$ et la solution $\overline{0}$-monomiale minimale n'est pas irréductible.
}
\end{remarks}

Cette définition peut sembler un peu artificielle mais elle constitue en fait une généralisation intéressante des deux concepts précédents comme le montre le résultat ci-dessous : 

\begin{theorem}
\label{211}

Soit $N \geq 3$.
\\
\\i) $N$ impair. $N$ est semi monomialement irréductible si et seulement si $N$ est la puissance d'un nombre premier impair, c'est-à-dire $N$ est semi monomialement irréductible si et seulement si $N$ est quasi monomialement irréductible.
\\
\\ii) Si $N=2p^{n}$ avec $p \in \mathbb{P}$ et $n \geq 1$ ou si $N=2\times 3^{a}5^{b}$ avec $a,b \geq 0$ et $a+b \neq 0$ alors $N$ est semi monomialement irréductible.
\\
\\iii) Si $N=2^{a}3^{b}5^{c}7^{d}17^{e}31^{f}127^{g}$ avec $a \geq 2$ et $(b,c,d,e,f,g) \in \mathbb{N}^{6}$ alors $N$ est semi monomialement irréductible.

\end{theorem}

Ce théorème sera prouvé dans la section \ref{semi}. En particulier, on constate, en combinant ce résultat au théorème \ref{27}, que les entiers quasi monomialement irréductibles sont semi monomialement irréductibles. Ainsi, la notion d'entier semi monomialement irréductible généralise celle d'entier quasi monomialement irréductible. En revanche, on ne dispose pas d'une classification complète de ces entiers dans le cas pair.

\section{Classification des entiers monomialement irréductibles et quasi monomialement irréductibles}
\label{quasi}

L'objectif de cette section est de démontrer les théorèmes \ref{26} et \ref{27} et de présenter les résultats partiels obtenus précédemment.

\subsection{Résultats préliminaires}
\label{pre}

Avant de procéder à la preuve des théorèmes de classification présentés dans la section précédente, on a besoin de plusieurs résultats que l'on regroupe ici.
\\
\\ \indent Pour commencer, on commence par rappeler un certain nombre de résultats généraux, obtenus principalement dans \cite{M1, M2}, qui nous seront utiles dans toute la suite :
\begin{itemize}
\item $M_{n}(\overline{-k},\ldots,\overline{-k})=(-1)^{n}M_{n}(\overline{k},\ldots,\overline{k})^{T}$;
\item les solutions $\pm \overline{k}$-monomiales minimales de \eqref{p} ont la même taille;
\item la solution $\overline{k}$-monomiale minimale de \eqref{p} est irréductible si et seulement si la solution $\overline{-k}$-monomiale minimale de \eqref{p} est irréductible;
\item les solutions $\pm \overline{1}$-monomiales minimales de \eqref{p} sont irréductibles et ce sont les seules solutions de taille 3;
\item les solutions de \eqref{p} de taille 4 sont de la forme $(\overline{a},\overline{b},\overline{a},\overline{b})$ avec $\overline{ab}=\overline{2}$ et $(\overline{-a},\overline{b},\overline{a},\overline{-b})$ avec $\overline{ab}=\overline{0}$;
\item si $N$ est pair, la solution $\overline{\frac{N}{2}}$-monomiale minimale de \eqref{p} est irréductible.
\\
\end{itemize}

\noindent De plus, si $(\overline{a},\overline{k},\ldots,\overline{k},\overline{b})$ est une solution de \eqref{p} alors $\overline{a}=\overline{b}$ et $\overline{a}(\overline{a}-\overline{k})=\overline{0}$ (voir \cite{M1} proposition 3.15).

\begin{lemma}
\label{30}

Soient $N$ un entier supérieur à 2 et $\overline{k}$ un élément de $\mathbb{Z}/N\mathbb{Z}$. Soit $d$ un diviseur de $N$. La taille de la solution $\overline{k}$-monomiale minimale de \eqref{p} est un multiple de la taille de la solution $(k+d\mathbb{Z})$-monomiale minimale de $(E_{d})$.

\end{lemma}

\begin{proof}

Soit $r$ la taille de la solution $\overline{k}$-monomiale minimale de \eqref{p}. Il existe $\epsilon \in \{-1, 1\}$ tel que $M_{r}(\overline{k},\ldots,\overline{k})=\overline{\epsilon} Id$. Comme $d$ divise $N$, $M_{r}(k+d\mathbb{Z},\ldots,k+d\mathbb{Z})=(\epsilon+d\mathbb{Z}) Id$. De plus, la taille de la solution $(k+d\mathbb{Z})$-monomiale minimale de $(E_{d})$ est l'ordre de $M_{1}(k+d\mathbb{Z})$ dans $PSL_{2}(\mathbb{Z}/d \mathbb{Z})$. Donc, $r$ est un multiple de la taille de la solution $(k+d\mathbb{Z})$-monomiale minimale de $(E_{d})$.

\end{proof}

\begin{lemma}
\label{31}

Soient $N$ un entier supérieur à 2 et $\overline{k}$ un élément de $\mathbb{Z}/N\mathbb{Z}$. Soit $n$ la taille de la solution $\overline{k}$-monomiale minimale de \eqref{p}. 
\\
\\A) Soit $m$ un entier naturel non nul.
\\
\\i) Si $(\overline{a},\overline{k},\ldots,\overline{k},\overline{b}) \in (\mathbb{Z}/N\mathbb{Z})^{nm}$ est une solution de \eqref{p} alors $\overline{a}=\overline{b}=\overline{k}$.
\\
\\ii) Il n'y a pas de solution de \eqref{p} de la forme $(\overline{a},\overline{k},\ldots,\overline{k},\overline{b})$ de taille $nm+1$.
\\
\\iii) Si $(\overline{a},\overline{k},\ldots,\overline{k},\overline{b}) \in (\mathbb{Z}/N\mathbb{Z})^{nm+2}$ est une solution de \eqref{p} alors $\overline{a}=\overline{b}=\overline{0}$.
\\
\\B) On suppose que la solution $\overline{k}$-monomiale minimale de \eqref{p} est irréductible. Le $l$-uplet $(\overline{a},\overline{k},\ldots,\overline{k},\overline{a}) \in (\mathbb{Z}/N\mathbb{Z})^{l}$ est une solution de \eqref{p} si et seulement si une deux conditions suivantes est vérifiée :
\begin{itemize}
\item $l \equiv 0 [n]$ et $\overline{a}=\overline{k}$;
\item $l \equiv 2 [n]$ et $\overline{a}=\overline{0}$.

\end{itemize}

\end{lemma}

\begin{proof}

A) Cela découle de calculs matriciels simples (voir \cite{M4} lemme 3.4 pour le détail de ces derniers).
\\
\\B) Il existe $\alpha \in \{-1, 1\}$ tel que $M_{n}(\overline{k},\ldots,\overline{k})=\overline{\alpha}Id$. Soit $(\overline{a},\overline{k},\ldots,\overline{k},\overline{a}) \in (\mathbb{Z}/N\mathbb{Z})^{l}$.
\\
\\Si $l \equiv 0 [n]$ et $\overline{a}=\overline{k}$ alors $M_{l}(\overline{a},\overline{k},\ldots,\overline{k},\overline{a})=M_{n}(\overline{k},\ldots,\overline{k})^{\frac{l}{n}}=\overline{\alpha}^{\frac{l}{n}}Id$. Si $l \equiv 2 [n]$ et $\overline{a}=\overline{0}$ alors $M_{l}(\overline{a},\overline{k},\ldots,\overline{k},\overline{a})=M_{1}(\overline{0})M_{n}(\overline{k},\ldots,\overline{k})^{\frac{l-2}{n}}M_{1}(\overline{0})=-(\overline{\alpha})^{\frac{l-2}{n}}Id$.
\\
\\Si $(\overline{a},\overline{k},\ldots,\overline{k},\overline{a}) \in (\mathbb{Z}/N\mathbb{Z})^{l}$ est une solution de \eqref{p}. Il existe $(q,r) \in \mathbb{N}^{2}$ tel que $l=nq+r$ avec $0 \leq r<n$. Mais aussi, par A), on a $r \neq 1$. Si $r=0$ alors $\overline{a}=\overline{k}$ et si $r=2$ alors $\overline{a}=\overline{0}$. Supposons $r \geq 3$. Il existe $\beta \in \{-1, 1\}$ tel que :
\begin{eqnarray*} 
\overline{\beta}Id &=& M_{l}(\overline{a},\overline{k},\ldots,\overline{k},\overline{a}) \\
  &=& M_{1}(\overline{a})M_{n}(\overline{k},\ldots,\overline{k})^{q}M_{r-2}(\overline{k},\ldots,\overline{k})M_{1}(\overline{a}) \\
	&=& \overline{\alpha}^{q} M_{r}(\overline{a},\overline{k},\ldots,\overline{k},\overline{a}).
\end{eqnarray*}	

\noindent Ainsi, le $r$-uplet $(\overline{a},\overline{k},\ldots,\overline{k},\overline{a})$ est une solution de \eqref{p}. Comme $r>2$, on peut utiliser la solution $(\overline{a},\overline{k},\ldots,\overline{k},\overline{a})$ pour réduire la solution $\overline{k}$-monomiale minimale de \eqref{p}. Comme la solution $\overline{k}$-monomiale minimale de \eqref{p} est irréductible, ceci est absurde.

\end{proof}

\noindent On dispose également des deux résultats assez puissants énoncés ci-dessous :

\begin{theorem}[\cite{M1} Théorème 2.6]
\label{32}

Soit $N \geq 3$. La solution $\overline{2}$-monomiale minimale de \eqref{p} est irréductible de taille $N$ avec \[M_{N}(\overline{2},\ldots,\overline{2})=Id.\]

\end{theorem}

\begin{theorem}[\cite{M4} Théorème 2.6]
\label{33}

Soient $p$ un nombre premier, $n$ un entier naturel non nul et $N=p^{n}$. Soit $k \in \mathbb{Z}$.
\\
\\i) Si $p \neq 2$. La solution $\overline{k}$-monomiale minimale de \eqref{p} est irréductible si et seulement si $p$ ne divise pas $k$. En particulier, \eqref{p} a $\varphi(p^{n})=p^{n-1}(p-1)$ solutions monomiales minimales irréductibles.
\\
\\ii) Si $p=2$. La solution $\overline{k}$-monomiale minimale de \eqref{p} est irréductible si et seulement si  une des propriétés suivantes est vérifiée :
\begin{itemize}
\item $k$ est impair; 
\item $\overline{k}=\overline{2^{n-1}}$;
\item $n \geq 2$ et il existe un entier impair $a$ tel que $k=2a$.
\end{itemize}

\noindent En particulier, si $n \geq 3$, \eqref{p} a $3 \times 2^{n-2}+1$ solutions monomiales minimales irréductibles.

\end{theorem} 

\subsection{Premiers éléments de classification des entiers monomialement irréductibles}

On va ici détailler les résultats que l'on avait déjà obtenus sur les entiers monomialement irréductibles. Pour commencer à classifier ces derniers, on avait notamment démontré le résultat intermédiaire ci-dessous :

\begin{proposition}[\cite{M3} propositions 3.7 et 3.8]
\label{34}

Soit $N \geq 2$. 
\begin{itemize}
\item Si $N$ est divisible par 16 alors la solution $\overline{\frac{N}{4}}$-monomiale minimale de \eqref{p} est réductible et $N$ est monomialement réductible.
\item Si $N$ est divisible par le carré d'un nombre premier impair $p$ alors la solution $\overline{\frac{N}{p}}$-monomiale minimale de \eqref{p} est réductible et $N$ est monomialement réductible.
\end{itemize}

\end{proposition}

Cette proposition, combinée à d'autres résultats, nous avait permis d'obtenir le théorème suivant qui donnait déjà un certain nombre d'éléments de classification.

\begin{theorem}[\cite{M3}, Théorème 2.7]
\label{35} 

Soit $N \geq 2$. On pose \[\Omega=\{107, 163, 173, 277, 283, 317, 347, 523, 557, 563, 613, 653, 733, 773, 787, 877, 907, 997\}.\]

\noindent i) Si $N$ est premier ou si $N \in \{4, 6, 8, 12, 24\}$ alors $N$ est monomialement irréductible.
\\
\\ii) On suppose $N$ pair. L'entier $N$ est monomialemnt irréductible si et seulement si $N \in \{2, 4, 6, 8, 12, 24\}$.
\\
\\iii) On suppose $N$ impair non premier. Si $N$ est monomialement irréductible alors $N$ est de la forme $\prod_{i=1}^{r} p_{i}$ où $r \geq 2$ et les $p_{i}$ sont des nombres premiers impairs deux à deux distincts vérifiant les conditions suivantes :
\begin{itemize}
\item $\forall i \in [\![1;r]\!]$, $p_{i} \equiv \pm 3 [5]$; 
\item $\forall i \in [\![1;r]\!]$, $p_{i} \equiv \pm 3 [8]$;
\item si $p_{i} \leq 1000$ alors $p_{i} \in \Omega$.
\end{itemize}

\end{theorem} 

\subsection{Démonstration du théorème \ref{26}}

Pour démontrer le théorème \ref{35} ii), on avait considéré le cas des entiers $N=km$ avec $m \equiv \pm l [k]$ et $l^{2} \equiv 1 [k]$. Dans ce cadre, on avait établi des formules pour le calcul de $K_{n}(\overline{lm+2})$ si $m \equiv -l [k]$ et $K_{n}(\overline{lm-2})$ si $m \equiv l [k]$, où $K_{n}(\overline{x})$ représente le coefficient de la première ligne et de la première colonne de $M_{n}(\overline{x},\ldots,\overline{x})$. Avec cela, on a pu considérer les cas des entiers $N=km$ avec $k \in \{2, 3, 4, 6, 8, 12, 24\}$ et $m$ impair non divisible par 3 (pour plus de détails voir \cite{M3} section 3.5).  Malheureusement, on ne peut pas généraliser telle quelle cette méthode. En effet, prenons, par exemple, $N=48 \times 5$. L'entier $5$ est impair et non divisible par 3 mais $5^{2} \not \equiv 1 [48]$. Pour terminer notre classification, on va donc considérer ce qui a été fait suivant un autre angle. En effet, dans ce qui précède, on a en réalité considéré $x \equiv \pm 2 [m]$ et $x \equiv \pm 1 [k]$. C'est cette approche qui va nous permettre d'obtenir un nouveau résultat. Une fois celui-ci établi, on n'obtiendra pas de formule pour le calcul des $K_{n}(\overline{x})$ mais on aura une solution monomiale minimale réductible non nulle.

\begin{proposition}
\label{36}

Soient $N=nm$ avec $n$ et $m$ des entiers naturels premiers entre eux différents de 1. On suppose que $m$ est impair et non divisible par 3. Il existe un entier $1 \leq k \leq N-1$ tel que $k$ est premier avec $N$ et tel que la solution $\overline{k}$-monomiale minimale de \eqref{p} est réductible, de taille $6m$ si $n>2$, et $3m$ sinon.

\end{proposition}

\begin{proof}

Les entiers $m$ et $n$ sont premiers entre eux, donc, par le théorème de Bézout, il existe $(a,b) \in \mathbb{Z}^{2}$ tels que $am+bn=1$. On pose $k=am+2bn$. En particulier, $k \equiv 1 [n]$ et $k \equiv 2 [m]$. L'entier $m$ est impair et non divisible par 3. Donc, $m \equiv 1, 2 [3]$, c'est-à-dire $m=1+3h$ ou $m=2+3h$ (avec $h \in \mathbb{N}^{*}$). De plus, $m$ est impair. Donc, si $m=1+3h$, $h$ est pair. Si $m=2+3h$, $h$ est impair.
\\
\\Supposons par l'absurde que $k$ et $N$ ne soient pas premiers entre eux. Il existe un nombre premier $p$ qui divise $k$ et $N=nm$. Par le lemme d'Euclide, $p$ divise $n$ ou $p$ divise $m$. Si $p$ divise $n$ alors puisque $k \equiv 1 [n]$, $p$ divise 1. Ceci est absurde. Donc, $p$ divise $m$. Puisque $k \equiv 2 [m]$, $p$ divise 2. Ainsi, $p=2$. Donc, $m$ est pair, ce qui est absurde. Ainsi, $k$ est premier avec $N$.
\\
\\On va maintenant calculer la taille de la solution $\overline{k}$-monomiale minimale de \eqref{p}.
\\
\\Par le théorème \ref{32}, la taille de la solution $(k+n\mathbb{Z})$-monomiale minimale de $(E_{n})$ est égale à 3 et la taille de la solution $(k+m\mathbb{Z})$-monomiale minimale de $(E_{m})$ est égale à $m$. De plus, on a 
\[M_{3}(k+n\mathbb{Z},k+n\mathbb{Z},k+n\mathbb{Z})=M_{3}(1+n\mathbb{Z},1+n\mathbb{Z},1+n\mathbb{Z})=(-1+n\mathbb{Z})Id;\]
\[M_{m}(k+m\mathbb{Z},\ldots,k+m\mathbb{Z})=M_{m}(2+m\mathbb{Z},\ldots,2+m\mathbb{Z})=(1+m\mathbb{Z})Id.\]

\noindent Comme $n$ et $m$ divisent $N$, la taille $l$ de la solution $\overline{k}$-monomiale minimale de \eqref{p} est un multiple de $3$ et de $m$ (lemme \ref{30}), c'est-à-dire un multiple de $3m$ (puisque $3$ et $m$ sont premiers entre eux). Supposons $n>2$. 
\[M_{3m}(k+n\mathbb{Z},\ldots,k+n\mathbb{Z})=(-1+n\mathbb{Z})^{m}Id=(-1+n\mathbb{Z})Id\neq (1+n\mathbb{Z})Id~(n>2).\] et 
\[M_{3m}(k+m\mathbb{Z},\ldots,k+m\mathbb{Z})=(1+m\mathbb{Z})^{3}Id=(1+m\mathbb{Z})Id \neq (-1+m\mathbb{Z})Id~(m>2);\]

\noindent Donc, $l \neq 3m$. En revanche, 
\[M_{6m}(k+n\mathbb{Z},\ldots,k+n\mathbb{Z})=(-1+n\mathbb{Z})^{2m}Id=(1+n\mathbb{Z})Id;\]
\[M_{6m}(k+m\mathbb{Z},\ldots,k+m\mathbb{Z})=(1+m\mathbb{Z})^{6}Id=(1+m\mathbb{Z})Id.\]

\noindent Par le lemme chinois ($m$ et $n$ premiers entre eux), on a $l=6m$. Si $n=2$ on montre de façon analogue que $l=3m$.
\\
\\On distingue maintenant deux cas :
\\
\\ \uwave{Premier cas :} On suppose $m=1+3h$ avec $h$ pair.
\\
\\Posons $x=am$. On a :
\begin{itemize}
\item $x \equiv 0 [m]$, donc, 
\begin{eqnarray*}
 M &=& M_{m+2}(x+m\mathbb{Z},k+m\mathbb{Z},\ldots,k+m\mathbb{Z},x+m\mathbb{Z}) \\
   &=& M_{m+2}(0+m\mathbb{Z},2+m\mathbb{Z},\ldots,2+m\mathbb{Z},0+m\mathbb{Z}) \\
	 &=& (-1+m\mathbb{Z})Id.
	\end{eqnarray*}

\item $x \equiv 1 [n]$, donc, 
\begin{eqnarray*}
\tilde{M} &=& M_{m+2}(x+n\mathbb{Z},k+n\mathbb{Z},\ldots,k+n\mathbb{Z},x+n\mathbb{Z}) \\
  &=& M_{m+2}(1+n\mathbb{Z},\ldots,1+n\mathbb{Z}) \\
	&=& (-1+n\mathbb{Z})^{h+1}Id \\
	&=& (-1+n\mathbb{Z})Id~~{\rm car}~h~{\rm est~pair}.
\end{eqnarray*}

\end{itemize}

\noindent Par le lemme chinois, $M_{m+2}(\overline{x},\overline{k},\ldots,\overline{k},\overline{x})=-Id$. Comme la solution $\overline{k}$-monomiale minimale de \eqref{p} est de taille $6m$ (ou $3m$), on peut utiliser le $(m+2)$-uplet $(\overline{x},\overline{k},\ldots,\overline{k},\overline{x})$ pour la réduire.
\\
\\ \uwave{Deuxième cas :} On suppose $m=2+3h$ avec $h$ impair.
\\
\\Posons $x=2bn$. On a :
\begin{itemize}
\item $x \equiv 2 [m]$, donc, 
\begin{eqnarray*}
M &=& M_{m}(x+m\mathbb{Z},k+m\mathbb{Z},\ldots,k+m\mathbb{Z},x+m\mathbb{Z}) \\
  &=& M_{m}(2+m\mathbb{Z},\ldots,2+m\mathbb{Z}) \\
	&=& (1+m\mathbb{Z})Id.
\end{eqnarray*}

\item $x \equiv 0 [n]$, donc, 
\begin{eqnarray*}
\tilde{M} &=& M_{m}(x+n\mathbb{Z},k+n\mathbb{Z},\ldots,k+n\mathbb{Z},x+n\mathbb{Z}) \\
  &=& M_{m}(0+n\mathbb{Z},1+n\mathbb{Z},\ldots,1+n\mathbb{Z},0+n\mathbb{Z}) \\
	&=& (-1+n\mathbb{Z})(-1+n\mathbb{Z})^{h}Id \\
	&=& (1+n\mathbb{Z})Id~~{\rm car}~h~{\rm est~impair}.
\end{eqnarray*}

\end{itemize}

\noindent Par le lemme chinois, $M_{m}(\overline{x},\overline{k},\ldots,\overline{k},\overline{x})=Id$. Comme la solution $\overline{k}$-monomiale minimale de \eqref{p} est de taille $6m$ (ou $3m$), on peut utiliser le $m$-uplet $(\overline{x},\overline{k},\ldots,\overline{k},\overline{x})$ pour la réduire.
\\
\\Ainsi, la solution $\overline{k}$-monomiale minimale de \eqref{p} est réductible de taille $6m$ (ou $3m$ si $n=2$).

\end{proof}

Cette proposition va nous être très utile pour obtenir les deux résultats de classification souhaités. De plus, cette preuve présente l'avantage d'être constructive.

\begin{examples}
{\rm On va détailler deux cas :
\begin{itemize}
\item $N=15=3 \times 5$, $m=5$ et $n=3$. Comme 5 divise $N$, on peut, bien sûr, utiliser les résultats du théorème \ref{35}, mais également ceux de la proposition précédente. On a $5\times (-1) +3 \times 2=1$. Posons $k=5 \times (-1)+2 \times 3 \times 2=7$. $m=2+3\times 1$ et on pose $x=12$. Par le résultat précédent, la solution $\overline{k}$-monomiale minimale de \eqref{p} est réductible. Elle est de taille $6m=30$ et on peut la réduire avec la solution $(\overline{-3},\overline{7},\overline{7},\overline{7},\overline{-3})$.
\\

\item Prenons maintenant le plus petit entier sur lequel on ne disposait pas encore d'information, c'est-à-dire $N=17~441=107 \times 163$, $n=107$ et $m=163$. On a $107 \times 32 +163 \times (-21)=1$. Posons $k=163 \times (-21)+2 \times 107 \times 32=-3423+6848=3425$. $163=1+3 \times 54$ (c'est-à-dire $h=54$) et on pose $x=163 \times (-21)=-3423$. Par le résultat précédent, la solution $\overline{k}$-monomiale minimale de \eqref{p} est réductible. Elle est de taille $6m=978$ et on peut la réduire avec la solution $(\overline{-3423},\overline{k},\ldots,\overline{k},\overline{-3423})$ de taille $l'=m+2=165$.

\end{itemize}
}
\end{examples}

\noindent On effectue maintenant la classification des entiers monomialement irréductibles.

\begin{proof}[Démonstration du théorème \ref{26}]

$\uwave{\Longleftarrow :}$ Si $N \in \mathbb{P} \cup \{4, 6, 8, 12, 24\}$ alors, par le théorème \ref{35} i), $N$ est monomialement irréductible.
\\
\\ $\uwave{\Longrightarrow :}$ Si $N$ est monomialement irréductible. 
\\
\\On suppose $N$ impair. Par la proposition \ref{36}, $N$ n'est pas de la forme $N=nm$ avec $n$ et $m$ des entiers naturels impairs premiers entre eux différents de 1. Donc, $N$ est la puissance d'un nombre premier. Par la proposition \ref{34}, $N$ n'a pas de facteur carré. Donc, $N$ est premier. 
\\
\\On suppose $N$ pair. Il existe $(a,m) \in  (\mathbb{N}^{*})^{2}$ et $b \in \mathbb{N}$ tels que $N=2^{a}3^{b}m$ et $m$ impair non divisible par 3. Si $m>1$ alors par la proposition \ref{36}, $N$ est monomialement réductible, ce qui est absurde. Donc, $m=1$. Par la proposition \ref{34}, $b \leq 1$ et $a \leq 3$, c'est-à-dire $N \in \{2, 4, 6, 8, 12, 24\}$.

\end{proof}

\begin{remark}
{\rm 
Pour obtenir une liste des entiers monomialement irréductibles, il suffit donc de consulter une liste de nombres premiers et d'y rajouter 4, 6, 8, 12 et 24. On peut par exemple regarder les tables 1 et 2 de
\cite{Ri} ou \cite{OEIS} A000040. Notons par ailleurs que les entiers monomialement irréductibles sont également
référencés dans \cite{OEIS} (A350242).
}
\end{remark}

Notons que dans cette preuve, on n'a pas utilisé les points ii) et iii) du théorème \ref{35}. Toutefois, les éléments qui nous avaient permis de démontrer ces deux points gardent un certain intérêt. En effet, pour obtenir ii), on avait, comme on l'a déjà indiqué, établi des formules pour le calcul des $K_{n}(\overline{k})$. Quant au point iii), on a démontré pour l'obtenir des résultats de réduction utilisant des solutions de très petite taille. Par exemple, si $N$ est divisible par un nombre premier congru à $\pm 1 [5]$ alors on avait établi que l'on peut trouver une solution monomiale minimale de \eqref{p} que l'on peut réduire avec une solution de taille 5. Ainsi, si $N=16~456~333=4003\times 4111$, on peut trouver un $k$ tel que la solution $\overline{k}$-monomiale minimale  de \eqref{p} peut être réduite avec une solution de taille 5 alors que la proposition \ref{36} fournit un $l$ tel que la solution $\overline{l}$-monomiale minimale peut être réduite avec une solution de taille supérieure à 4003.

\subsection{Démonstration du théorème \ref{27}}

On va maintenant s'intéresser aux entiers quasi monomialement irréductibles. Grâce aux résultats déjà établis, il nous reste en fait uniquement à considérer les cas des entiers de la forme $2^{n}3^{m}$.

\begin{proposition}
\label{37}

Les entiers de la forme $N=2^{n}3^{m} \geq 2$ sont quasi monomialement irréductibles.

\end{proposition}

\begin{proof}

On considère différents cas en fonction de $n$.
\\
\\ $\uwave{n=0 :}$ Par le théorème \ref{33}, la proposition est vraie.
\\
\\ $\uwave{n=1 :}$ Si $m=0$, le résultat est vrai (Théorème \ref{33}). On suppose donc $m \geq 1$. Soit $k$ un entier premier avec $N=2 \times 3^{m}$. 
\\
\\Notons $l$ la taille de la solution $(k+3^{m} \mathbb{Z})$-monomiale minimale de $(E_{3^{m}})$ et $r$ la taille de la solution $\overline{k}$-monomiale minimale de \eqref{p}. Comme $3^{m}$ divise $N$, $r$ est un multiple de $l$. La solution $(k+3 \mathbb{Z})$-monomiale minimale de $(E_{3})$ est de taille 3, car $k$ n'est pas un multiple de 3. Donc, comme 3 divise $3^{m}$, $l$ est un multiple de $3$. 
\\
\\Il existe $\epsilon \in \{1, -1\}$ tel que $M_{l}(k+3^{m} \mathbb{Z},\ldots,k+3^{m} \mathbb{Z})=(\epsilon +3^{m}\mathbb{Z}) Id$. On a :

\begin{eqnarray*}
M_{l}(k+2 \mathbb{Z},\ldots,k+2 \mathbb{Z}) &=& M_{l}(1+2 \mathbb{Z},\ldots,1+2 \mathbb{Z}) \\
                                            &=& (-1+2\mathbb{Z})^{\frac{l}{3}}Id \\
																						&=& (\epsilon+2\mathbb{Z}) Id~~({\rm car}~-1 \equiv 1 [2]).
\end{eqnarray*}																						

\noindent Par le lemme chinois, $r=l$.
\\
\\Supposons par l'absurde que la solution $\overline{k}$-monomiale minimale de \eqref{p} est réductible. Il existe une solution $(\overline{x},\overline{k},\ldots,\overline{k},\overline{x})$ de \eqref{p} de taille $3 \leq l' \leq l-1$. 
\\
\\Comme $3^{m}$ divise $N$, le $l'$-uplet $(x+3^{m} \mathbb{Z},k+3^{m} \mathbb{Z},\ldots,k+3^{m} \mathbb{Z},x+3^{m} \mathbb{Z})$ est une solution de $(E_{3^{m}})$. Donc, la solution $(k+3^{m} \mathbb{Z})$-monomiale minimale de $(E_{3^{m}})$ est réductible. Or, comme $k$ n'est pas divisible par 3, celle-ci est irréductible (Théorème \ref{33}). On arrive ainsi à une absurdité.
\\
\\Ainsi, la solution $\overline{k}$-monomiale minimale de \eqref{p} est irréductible.
\\
\\ $\uwave{n \geq 2 :}$ Si $m=0$, le résultat est vrai (Théorème \ref{33}). On suppose donc $m \geq 1$. Soit $k$ un entier premier avec $N=2^{n} \times 3^{m}$. Par le théorème de Bézout, il existe $(a,b) \in \mathbb{Z}^{2}$ tels que $a2^{n}+b3^{m}=1$. En particulier, $a$ n'est pas divisible par 3 et $b$ est impair.
\\
\\Notons $r$ la taille de la solution $\overline{k}$-monomiale minimale de \eqref{p}. Supposons par l'absurde que cette dernière est réductible. Il existe une solution $(\overline{x},\overline{k},\ldots,\overline{k},\overline{x})$ de \eqref{p} de taille $3 \leq d \leq r-1$ permettant de la réduire. On va considérer plusieurs diviseurs de $N$.

\begin{itemize}
\item $2^{n}$ divise $N$. Le $d$-uplet $(x+2^{n} \mathbb{Z},k+2^{n} \mathbb{Z},\ldots,k+2^{n} \mathbb{Z},x+2^{n} \mathbb{Z})$ est une solution de $(E_{2^{n}})$. Or, comme $k$ est impair, la solution $(k+2^{n} \mathbb{Z})$-monomiale minimale de $(E_{2^{n}})$ est irréductible, par le théorème \ref{33}. Par le lemme \ref{31}, $x \equiv 0, k [2^{n}]$.
\\
\item $3^{m}$ divise $N$. Le $d$-uplet $(x+3^{m} \mathbb{Z},k+3^{m} \mathbb{Z},\ldots,k+3^{m} \mathbb{Z},x+3^{m} \mathbb{Z})$ est une solution de $(E_{3^{m}})$. Or, comme $k$ n'est pas divisible par 3, la solution $(k+3^{m} \mathbb{Z})$-monomiale minimale de $(E_{3^{m}})$ est irréductible, par le théorème \ref{33}. Par le lemme \ref{31}, $x \equiv 0, k [3^{m}]$.
\\
\item $h=2 \times 3^{m}$ divise $N$. Le $d$-uplet $(x+h \mathbb{Z},k+h \mathbb{Z},\ldots,k+h \mathbb{Z},x+h \mathbb{Z})$ est une solution de $(E_{h})$. Or, comme $k$ est premier avec $h$, la solution $(k+h \mathbb{Z})$-monomiale minimale de $(E_{h})$ est irréductible (point précédent). Par le lemme \ref{31}, $x \equiv 0, k [h]$.
\\
\end{itemize}

\noindent Si $x \equiv 0 [2^{n}]$ et $x \equiv 0 [3^{m}]$ alors, par le lemme chinois, $\overline{x}=\overline{0}$. Dans ce cas, le $d-2$-uplet $(\overline{k},\ldots,\overline{k})$ est une solution de \eqref{p}, ce qui est absurde (puisque $d-2 < r$). Si $x \equiv k [2^{n}]$ et $x \equiv k [3^{m}]$ alors, par le lemme chinois, $\overline{x}=\overline{k}$. Dans ce cas, le $d$-uplet $(\overline{k},\ldots,\overline{k})$ est une solution de \eqref{p}, ce qui est absurde (puisque $d < r$). Donc, on a $x \equiv 0 [2^{n}]$ et $x \equiv k [3^{m}]$ ou $x \equiv k [2^{n}]$ et $x \equiv 0 [3^{m}]$.
\\
\\Si $x \equiv 0 [2^{n}]$ et $x \equiv k [3^{m}]$. Par le lemme chinois, on a $x \equiv ak2^{n} [N]$. Si $x \equiv 0 [h]$ alors 3 divise $ak2^{n}$, ce qui est absurde. Donc, $x \equiv k [h]$ et 2 divise $k(a2^{n}-1)$. Or, ceci est absurde, car $k$ et $(a2^{n}-1)$ sont impairs.
\\
\\Si $x \equiv k [2^{n}]$ et $x \equiv 0 [3^{m}]$. Par le lemme chinois, on a $x \equiv bk3^{m} [N]$. Si $x \equiv 0 [h]$ alors 2 divise $bk3^{m}$, ce qui est absurde. Donc, $x \equiv k [h]$ et 3 divise $k(b3^{m}-1)$. Or, ceci est absurde, car $k$ et $(b3^{m}-1)$ ne sont pas divisibles par 3.
\\
\\Ainsi, la solution $\overline{k}$-monomiale minimale de \eqref{p} est irréductible.

\end{proof}

\noindent Avec ce résultat, on peut effectuer la classification des entiers quasi monomialement irréductibles.

\begin{proof}[Démonstration du théorème \ref{27}]

$\uwave{\Longleftarrow :}$ Si $N \in \{p^{n}, p \in \mathbb{P}, n \in \mathbb{N}^{*}\}$ alors, par le théorème \ref{33}, $N$ est quasi monomialement irréductible. Si $N \in \{2^{n}3^{m}, (n,m) \in (\mathbb{N}^{*})^{2}\}$ alors, par la proposition précédente, $N$ est quasi monomialement irréductible.
\\
\\ $\uwave{\Longrightarrow :}$ Si $N$ est quasi monomialement irréductible.
\\
\\Si $N$ est impair. Par la proposition \ref{36}, $N$ n'est pas de la forme $N=nm$ avec $n$ et $m$ des entiers naturels impairs premiers entre eux différents de 1. Donc, $N$ est la puissance d'un nombre premier impair.
\\
\\Si $N$ est pair. Il existe $(a,m) \in  (\mathbb{N}^{*})^{2}$ et $b \in \mathbb{N}$ tels que $N=2^{a}3^{b}m$ et $m$ impair non divisible par 3. Par la proposition \ref{36}, on a nécessairement $m=1$, c'est-à-dire $N \in \{2^{n}3^{m}, n \in \mathbb{N}^{*}, m \in \mathbb{N}\}$.

\end{proof}

Dans l'annexe \ref{A}, on donne la liste, dans l'ordre croissant, de tous les entiers quasi monomialement irréductibles inférieurs à 1000.
\\
\\ \indent Après la définition \ref{25}, on avait indiqué qu'il était possible, pour $N$ quasi monomialement   irréductible, d'avoir un $k$ non premier avec $N$ tel que la solution $\overline{k}$-monomiale minimale de \eqref{p} est irréductible. À la lueur du théorème de classification, on peut en fait savoir exactement les $N$ pour lesquels une telle chose est possible. 

\begin{proposition}
\label{38}

Soit $N \geq 2$. Notons $\Omega_{N}$ l'ensemble des entiers $k$ appartenant à $[\![0;N-1]\!]$ pour lesquels la solution $\overline{k}$-monomiale minimale de \eqref{p} est irréductible. L'ensemble $\Omega_{N}$ est égal à l'ensemble $\{k \in [\![0;N-1]\!], k~{\rm premier~avec}~N\}$ si et seulement si $N=2$ ou $N=p^{n}$ avec $p$ premier impair et $n \geq 1$.

\end{proposition}

\begin{proof}

Si $N=2$ ou $N=p^{n}$ avec $p$ premier impair alors, par le théorème \ref{33}, on a \[\Omega_{N}=\{k \in [\![0;N-1]\!], k~{\rm premier~avec}~N\}.\]

\noindent Si $\Omega_{N}=\{k \in [\![0;N-1]\!], k~{\rm premier~avec}~N\}$. Dans ce cas, $N$ est quasi monomialement irréductible. Par le théorème \ref{27}, $N \in \{p^{n}, p \in \mathbb{P}, n \in \mathbb{N}^{*}\} \cup \{2^{n}3^{m}, (n,m) \in (\mathbb{N}^{*})^{2}\}$. Si $N=2^{n}$ avec $n \geq 2$ ou si $N=2^{n}3^{m}$ avec $nm \geq 1$ alors $\frac{N}{2}$ n'est pas premier avec $N$ et pourtant la solution $\overline{\frac{N}{2}}$-monomiale minimale de \eqref{p} est irréductible (voir section \ref{pre}). Donc, nécessairement, $N=2$ ou $N=p^{n}$ avec $p$ premier impair et $n \geq 1$.

\end{proof}

Dans le cas où $N$ n'est pas une puissance d'un nombre premier, il y a donc plus que $\varphi(N)$ solutions monomiales minimales irréductibles. On fournit dans l'annexe \ref{B} le nombre de ces dernières lorsque $N$ est de la forme $2^{n}3^{m}$ avec $nm \geq 1$.

\section{Classification des solutions monomiales minimales de $(E_{2\times 3^{m}})$}
\label{class}

Avant de faire la preuve du théorème de classification, on rappelle le résultat ci-dessous qui précise l'énoncé du théorème \ref{33}.

\begin{lemma}[\cite{M4}, proposition 3.12]
\label{41}

Soit $p \in \mathbb{P}$ impair, $n \geq 2$, $1 \leq k \leq n-1$ et $a$ un entier premier avec $p$. La solution $\overline{ap^{k}}$-monomiale minimale de $(E_{p^{n}})$ est de taille $2p^{n-k}$ avec $M_{2p^{n-k}}(\overline{ap^{k}},\ldots,\overline{ap^{k}})=-Id$. On peut réduire cette dernière avec la solution $M_{4\times p^{n-k-1}}(\overline{ap^{k}-2ap^{n-1}},\overline{ap^{k}},\ldots,\overline{ap^{k}},\overline{ap^{k}-2ap^{n-1}})=Id$.

\end{lemma}

\noindent On peut maintenant démontrer le résultat annoncé :

\begin{proof}[Démonstration du théorème \ref{28}]

Si $k$ est premier avec $N$ alors la solution $\overline{k}$-monomiale minimale de \eqref{p} est irréductible (voir théorème \ref{27}). On s'intéresse donc maintenant aux cas où $k$ n'est pas premier avec $N$.
\\
\\ \underline{On suppose $k=2a$ avec $a$ non divisible par 3 :} 
\\
\\Notons $l$ la taille de la solution $(k+3^{m} \mathbb{Z})$-monomiale minimale de $(E_{3^{m}})$ et $r$ la taille de la solution $\overline{k}$-monomiale minimale de \eqref{p}. Comme $3^{m}$ divise $N$, $r$ est un multiple de $l$. Il existe $\epsilon \in \{1, -1\}$ tel que $M_{l}(k+3^{m} \mathbb{Z},\ldots,k+3^{m} \mathbb{Z})=(\epsilon +3^{m}\mathbb{Z}) Id$. De plus, la solution $(k+2 \mathbb{Z})$-monomiale minimale de $(E_{2})$ est de taille 2. Ainsi, $r$ est un multiple de ${\rm ppcm}(2,l)$.
\\
\\\uwave{On suppose $l$ pair.} ${\rm ppcm}(2,l)=l$ et on a :
\[M_{l}(k+3^{m} \mathbb{Z},\ldots,k+3^{m} \mathbb{Z})=(\epsilon +3^{m}\mathbb{Z}) Id;\]
\[M_{l}(k+2 \mathbb{Z},\ldots,k+2 \mathbb{Z})=M_{l}(0+2 \mathbb{Z},\ldots,0+2 \mathbb{Z})=(\epsilon+2\mathbb{Z}) Id.\]

\noindent Par le lemme chinois, $r=l$. Supposons par l'absurde que la solution $\overline{k}$-monomiale minimale de \eqref{p} est réductible. Il existe une solution $(\overline{x},\overline{k},\ldots,\overline{k},\overline{x})$ de \eqref{p} de taille $3 \leq l' \leq l-1$. Comme $3^{m}$ divise $N$, le $l'$-uplet $(x+3^{m} \mathbb{Z},k+3^{m} \mathbb{Z},\ldots,k+3^{m} \mathbb{Z},x+3^{m} \mathbb{Z})$ est une solution de $(E_{3^{m}})$. Donc, la solution $(k+3^{m} \mathbb{Z})$-monomiale minimale de $(E_{3^{m}})$ est réductible. Or, comme $k$ n'est pas divisible par 3, celle-ci est irréductible (Théorème \ref{33}). On arrive ainsi à une absurdité.
\\
\\ \uwave{On suppose $l$ impair.} ${\rm ppcm}(2,l)=2l$ et on a :
\[M_{2l}(k+3^{m} \mathbb{Z},\ldots,k+3^{m} \mathbb{Z})=(\epsilon +3^{m}\mathbb{Z})^{2} Id=(1+3^{m}\mathbb{Z}) Id;\]
\[M_{2l}(k+2 \mathbb{Z},\ldots,k+2 \mathbb{Z})=M_{2l}(0+2 \mathbb{Z},\ldots,0+2 \mathbb{Z})=(1+2\mathbb{Z}) Id.\]

\noindent Par le lemme chinois, $r=2l$. Supposons par l'absurde que la solution $\overline{k}$-monomiale minimale de \eqref{p} est réductible. Il existe une solution $(\overline{x},\overline{k},\ldots,\overline{k},\overline{x})$ de \eqref{p} de taille $3 \leq l' \leq l+1$. 
\\
\\Si $l' \leq l-1$ alors, en procédant comme dans le cas précédent, on a arrive à une absurdité. Ainsi, $l' \in \{l, l+1\}$. Si $l'=l+1$. Comme $3^{m}$ divise $N$, le $(l+1)$-uplet $(x+3^{m} \mathbb{Z},k+3^{m} \mathbb{Z},\ldots,k+3^{m} \mathbb{Z},x+3^{m} \mathbb{Z})$ est une solution de $(E_{3^{m}})$ ce qui est impossible puisque la solution $(k+3^{m} \mathbb{Z})$-monomiale minimale de $(E_{3^{m}})$ est de taille $l$ (voir lemme \ref{31}). Donc, $l'=l$. Comme $2$ divise $N$, le $l$-uplet $(x+2 \mathbb{Z},k+2 \mathbb{Z},\ldots,k+2 \mathbb{Z},x+2 \mathbb{Z})$ est une solution de $(E_{2})$ ce qui est impossible puisque la solution $(k+2 \mathbb{Z})$-monomiale minimale de $(E_{2})$ est de taille $2$ et $l$ est impair (voir lemme \ref{31}).
\\
\\Ainsi, la solution $\overline{k}$-monomiale minimale de \eqref{p} est irréductible.
\\
\\ \underline{On suppose $k=3^{m}$ :}
\\
\\Dans ce cas, $k=\frac{N}{2}$ et la solution $\overline{k}$-monomiale minimale de \eqref{p} est irréductible (voir section \ref{pre}).
\\
\\ \underline{On suppose $k=3^{m-1}$ ou $k=5 \times 3^{m-1}$ :}
\\
\\ $k=3^{m-1}=\frac{N}{6}$. En particulier, $k$ est impair, $k \equiv 3[6]$ et $\overline{6k}=\overline{0}$. On a :
\[M_{6}(\overline{k},\ldots,\overline{k}) = \begin{pmatrix}
   \overline{k^{6}-5k^{4}+6k^{2}-1}   & \overline{-k^{5}+4k^{3}-3k} \\
   \overline{k^{5}-4k^{3}+3k} & \overline{-k^{4}+3k^{2}-1}
\end{pmatrix}.\]

\noindent Or, $\overline{-k^{4}+3k^{2}}=\overline{k^{2}(-k^{2}+3)}$ et $k^{2} \equiv 3 [6]$. Ainsi, $-k^{2}+3$ est un multiple de 6 et $\overline{-k^{4}+3k^{2}}=\overline{0}$. De même, $\overline{k^{5}-4k^{3}+3k}=\overline{k^{5}+2k^{3}+3k}=\overline{k(k^{4}+2k^{2}+3)}$. $k^{4} \equiv 3 [6]$ et donc $k^{4}+2k^{2}+3 \equiv 0 [6]$, c'est-à-dire $\overline{k^{5}-4k^{3}+3k}=\overline{0}$.
\\
\\Comme $M_{6}(\overline{k},\ldots,\overline{k})$ est de déterminant $\overline{1}$, on a $M_{6}(\overline{k},\ldots,\overline{k})=-Id$. Ainsi, la taille $r$ de la solution $\overline{k}$-monomiale minimale divise 6, c'est-à-dire que celle-ci est égale à 1, 2, 3 ou 6. Or, \eqref{p} n'a pas de solution de taille 1 et $\overline{k} \neq \overline{0}, \overline{\pm 1}$. Donc, $r=6$. 
\\
\\Si la solution $\overline{k}$-monomiale minimale est réductible alors elle est la somme de deux solutions de taille 4 (puisqu'elle ne contient pas $\pm \overline{1}$). Dans ce cas, \eqref{p} a une solution de la forme $(\overline{a},\overline{k},\overline{k},\overline{a})$ avec $\overline{a} \neq \overline{k}$ (sinon la solution minimale serait de taille 4). Donc, on a $\overline{k^{2}}=\overline{0}$ et $\overline{a}=\overline{-k}$. Ainsi, Il existe un entier $l$ tel que $k^{2}=6lk$. Donc, $k^{2}$ est pair ce qui implique $k$ pair, ce qui est absurde. 
\\
\\Donc, la solution $\overline{k}$-monomiale minimale est irréductible. De plus, $\overline{-3^{m-1}}=\overline{5 \times 3^{m-1}}$, donc la solution $\overline{5 \times 3^{m-1}}$-monomiale minimale de \eqref{p} est irréductible.
\\
\\ \underline{On suppose $k=2 \times 3^{m-1}$ ou $k=4 \times 3^{m-1}$ :}
\\
\\$\frac{N}{3}=2 \times 3^{m-1}$ et 9 divise $N$ donc, par la proposition \ref{34}, la solution $\overline{2 \times 3^{m-1}}$-monomiale minimale de \eqref{p} est réductible. De plus, $\overline{-2 \times 3^{m-1}}=\overline{4 \times 3^{m-1}}$, donc la solution $\overline{4 \times 3^{m-1}}$-monomiale minimale de \eqref{p} est réductible.
\\
\\ \underline{On suppose $m \geq 3$ et $k=3^{l}a$ avec $1 \leq l \leq m-2$ et $a$ impair non divisible par 3 :} 
\\
\\La solution $(k+3^{m} \mathbb{Z})$-monomiale minimale de $(E_{3^{m}})$ est de taille $2\times 3^{m-l}$. La solution $(k+2 \mathbb{Z})$-monomiale minimale de $(E_{2})$ est de taille $3$. En procédant comme précédemment, la taille $r$ de la solution $\overline{k}$-monomiale minimale de \eqref{p} est un multiple de ${\rm ppcm}(3,2\times 3^{m-l})=2\times 3^{m-l}$. Comme $M_{2\times 3^{m-l}}(k+3^{m} \mathbb{Z},\ldots,k+3^{m} \mathbb{Z})=(-1 +3^{m}\mathbb{Z}) Id$ et $M_{2\times 3^{m-l}}(k+2 \mathbb{Z},\ldots,k+2 \mathbb{Z})=(-1+2 \mathbb{Z})Id$ (car $k$ impair et $m-l \geq 1$). Donc, $r=2\times 3^{m-l}$.
\\
\\ On a, par le lemme \ref{41} :
\[M_{4\times 3^{m-l-1}}(a3^{l}-2a3^{m-1}+3^{m}\mathbb{Z},k+3^{m}\mathbb{Z},\ldots,k+3^{m}\mathbb{Z},a3^{l}-2a3^{m-1}+3^{m}\mathbb{Z})=(1+3^{m}\mathbb{Z});\]

\noindent De plus, on a :
\begin{eqnarray*}
M &=& M_{4\times 3^{m-l-1}}(a3^{l}-2a3^{m-1}+2\mathbb{Z},k+2\mathbb{Z},\ldots,k+2\mathbb{Z},a3^{l}-2a3^{m-1}+2\mathbb{Z}) \\
  &=& M_{4\times 3^{m-l-1}}(1+2\mathbb{Z},\ldots,1+2\mathbb{Z}) \\
	&=& (1+2\mathbb{Z})^{4 \times 3^{m-l-2}}Id~~({\rm car}~l \leq m-2) \\
	&=& (1+2\mathbb{Z})Id.
\end{eqnarray*}

\noindent Par le lemme chinois, $M_{4\times 3^{m-l-1}}(\overline{a3^{l}-2a3^{m-1}},\overline{k},\ldots,\overline{k},\overline{a3^{l}-2a3^{m-1}})=Id$. 
\\
\\Comme $4\times 3^{m-l-1}< 2\times 3^{m-l}$, on peut réduire la solution $\overline{k}$-monomiale minimale de \eqref{p}.
\\
\\ \underline{On suppose $m \geq 3$ et $k=3^{l}a$ avec $1 \leq l \leq m-2$ et $a$ pair non divisible par 3 :} 
\\
\\La solution $(k+3^{m} \mathbb{Z})$-monomiale minimale de $(E_{3^{m}})$ est de taille $2\times 3^{m-l}$. La solution $(k+2 \mathbb{Z})$-monomiale minimale de $(E_{2})$ est de taille $2$. En procédant comme précédemment, on obtient que la taille $r$ de la solution $\overline{k}$-monomiale minimale est égale à $2\times 3^{m-l}$.
\\
\\ On a, par le lemme \ref{41} :
\[M_{4\times 3^{m-l-1}}(a3^{l}-2a3^{m-1}+3^{m}\mathbb{Z},k+3^{m}\mathbb{Z},\ldots,k+3^{m}\mathbb{Z},a3^{l}-2a3^{m-1}+3^{m}\mathbb{Z})=(1+3^{m}\mathbb{Z});\]

\noindent De plus, on a :
\begin{eqnarray*}
M &=& M_{4\times 3^{m-l-1}}(a3^{l}-2a3^{m-1}+2\mathbb{Z},k+2\mathbb{Z},\ldots,k+2\mathbb{Z},a3^{l}-2a3^{m-1}+2\mathbb{Z}) \\
  &=& M_{4\times 3^{m-l-1}}(0+2\mathbb{Z},\ldots,0+2\mathbb{Z}) \\
	&=& (1+2\mathbb{Z})^{2\times 3^{m-l-1}}Id \\
	&=& (1+2\mathbb{Z})Id.
\end{eqnarray*}

\noindent Par le lemme chinois, $M_{4\times 3^{m-l-1}}(\overline{a3^{l}-2a3^{m-1}},\overline{k},\ldots,\overline{k},\overline{a3^{l}-2a3^{m-1}})=Id$. 
\\
\\Comme on a $4\times 3^{m-l-1}< 2\times 3^{m-l}$, on peut réduire la solution $\overline{k}$-monomiale minimale de \eqref{p}.
\\
\\ \underline{Conclusion :} Ainsi, la solution $\overline{k}$-monomiale minimale de \eqref{p} est réductible si et seulement si $k=3^{l}a$ avec $1 \leq l \leq m-2$ ou $k=3^{m-1}a$ avec $a$ pair. Donc, la solution $\overline{k}$-monomiale minimale de \eqref{p} est réductible si et seulement si $k$ est un multiple de 3 différent de $3^{m}$, de $3^{m-1}$ et de $5 \times 3^{m-1}$. En particulier, \eqref{p} a $2 \times 3^{m-1}-3$ solutions monomiales minimales réductibles. Donc, \eqref{p} a \[2\times 3^{m}-(2 \times 3^{m-1}-3)=3^{m-1}(6-2)+3=4 \times 3^{m-1}+3\]
\noindent solutions monomiales minimales irréductibles.

\end{proof}

\begin{examples}
{\rm 
On donne ci-dessous deux applications du théorème \ref{28} :
\begin{itemize}
\item $N=18=2 \times 9$. La solution $\overline{k}$-monomiale minimale de $(E_{18})$ est réductible si et seulement si $\overline{k} \in \{\overline{0}, \overline{6}, \overline{12}\}$.
\item $N=54=2 \times 27$. La solution $\overline{k}$-monomiale minimale de $(E_{54})$ est réductible si et seulement si $\overline{k} \in \{\overline{0}, \overline{3}, \overline{6},\overline{12}, \overline{15}, \overline{18}, \overline{21}, \overline{24}, \overline{30}, \overline{33}, \overline{36}, \overline{39}, \overline{42}, \overline{48}, \overline{51}\}$.
\end{itemize}
}
\end{examples}

En reprenant les méthodes développées ci-dessus, on pourrait établir des résultats similaires pour différents multiples de 2. Toutefois, les démonstrations contiendraient des disjonctions de cas de plus en plus nombreuses et de plus en plus longues. Aussi, étant donné qu'il ne serait pas très intéressant de disposer d'un petit nombre de résultats de classification séparés les uns des autres, on se contente du théorème établi ci-dessus et de la formulation du problème général suivant :

\begin{pro}

Soit $N=2^{n}3^{m}$ avec $(n,m) \in (\mathbb{N}^{*})^{2}$. Classifier les solutions monomiales minimales irréductibles de \eqref{p}.

\end{pro}

On donne dans l'annexe \ref{C} la liste des solutions monomiales minimales irréductibles de $(E_{2^{n}3^{m}})$ pour quelques valeurs de $n$ et de $m$.

\section{Entiers semi monomialement irréductibles}
\label{semi}

\noindent L'objectif de cette section est de démontrer le théorème \ref{211}.

\subsection{Le cas des entiers impairs}

Pour un entier $N$ impair, l'ensemble des éléments de la forme $\overline{2a}$ avec $a$ premier avec $N$ est exactement l'ensemble des $\overline{k}$ avec $k$ premier avec $N$. Aussi, la classification des entiers semi monomialement irréductibles impairs découle immédiatement de celle des entiers quasi monomialement irréductibles impairs et donc de la proposition \ref{36}. Toutefois, la construction effectuée pour démontrer cette proposition n'est pas la seule possible. Aussi, on donne ci-dessous une autre manière d'obtenir une solution réductible.

\begin{proposition}
\label{51}

Soient $N=nm$ avec $n$ et $m$ des entiers naturels impairs, premiers entre eux et différents de 1. Il existe un entier $1 \leq k \leq N-1$ tel que $k$ est premier avec $N$ et tel que la solution $\overline{k}$-monomiale minimale de \eqref{p} est réductible de taille $2nm$.

\end{proposition}

\begin{proof}

Puisque $m >1$ et $n>1$, on a $N \geq 3$. Par ailleurs, on peut supposer, sans perte de généralité, que $m>n$. 
\\
\\Les entiers $m$ et $n$ sont premiers entre eux, donc, par le théorème de Bézout, il existe $(a,b) \in \mathbb{Z}^{2}$ tels que $am+bn=1$. On considère $k=2am-2bn$. En particulier, $k \equiv 2 [n]$ et $k \equiv -2 [m]$.
\\
\\Par le théorème \ref{32}, la taille de la solution $(k+n\mathbb{Z})$-monomiale minimale de $(E_{n})$ est égale à $n$ et la taille de la solution $(k+m\mathbb{Z})$-monomiale minimale de $(E_{m})$ est égale à $m$. De plus, on a :
\[M_{n}(2+n\mathbb{Z},\ldots,2+n\mathbb{Z})=(1+n\mathbb{Z})Id;\]
\[M_{m}(-2+m\mathbb{Z},\ldots,-2+m\mathbb{Z})=((-1)^{m}+m\mathbb{Z})Id=(-1+m\mathbb{Z})Id.\]

\noindent Comme $n$ et $m$ divisent $N$, la taille $l$ de la solution $\overline{k}$-monomiale minimale de \eqref{p} est un multiple de $n$ et de $m$, c'est-à-dire un multiple de $mn$ (puisque $n$ et $m$ sont premiers entre eux). Or, 
\[M_{nm}(2+n\mathbb{Z},\ldots,2+n\mathbb{Z})=(1+n\mathbb{Z})^{m}Id=(1+n\mathbb{Z})Id \neq (-1+n\mathbb{Z})Id\] et 
\[M_{nm}(-2+m\mathbb{Z},\ldots,-2+m\mathbb{Z})=(-1+m\mathbb{Z})^{n}Id=(-1+m\mathbb{Z})Id \neq (1+m\mathbb{Z})Id.\]

\noindent Ainsi, $l \neq nm$. En revanche, on a
\[M_{2nm}(2+n\mathbb{Z},\ldots,2+n\mathbb{Z})=(1+n\mathbb{Z})^{2m}Id=(1+n\mathbb{Z})Id;\]
\[M_{2nm}(-2+m\mathbb{Z},\ldots,-2+m\mathbb{Z})=(-1+m\mathbb{Z})^{2n}Id=(1+m\mathbb{Z})Id.\]

\noindent Donc, par le lemme chinois ($m$ et $n$ premiers entre eux), on a $l=2nm$.
\\
\\On va maintenant démontrer que la solution $\overline{k}$-monomiale minimale de \eqref{p} est réductible. 
\\
\\On effectue la division euclidienne de $m$ par $n$. Il existe $(q,r) \in \mathbb{N}^{2}$ tel que $m=qn+r$ et $0 \leq r < n$. Le pgcd de $r$ et de $n$ est le même que celui de $m$ et $n$. Donc, $r$ et $n$ sont premiers entre eux. Ainsi, $r$ est inversible modulo $n$ et il existe un entier $u$ tel que $ru \equiv 1 [n]$. En particulier, on a $r(-2u) \equiv -2 [n]$. Donc, il existe $v \in [\![1;n-1]\!]$ tel que $rv+2 \equiv 0 [n]$.
\\
\\Si $v$ est impair alors on pose $w=v$. Si $v$ est pair alors on pose $w=v+n$. Dans les deux cas, $w$ est impair, $w \in [\![1;2n-1]\!]$ et $rw+2 \equiv 0 [n]$. On pose $l'=2+mw$ et $x=2am$. On a :
 
\begin{itemize}
\item $x \equiv 2 [n]$ et $l'=2+mw=2+(qn+r)w=(rw+2)+qn \equiv 0 [n]$. Donc, 
\[M_{l'}(x+n\mathbb{Z},k+n\mathbb{Z},\ldots,k+n\mathbb{Z},x+n\mathbb{Z})=M_{l'}(2+n\mathbb{Z},\ldots,2+n\mathbb{Z})=(1+n\mathbb{Z})Id.\]

\item $x \equiv 0 [m]$ et $l'=2+mw \equiv 2 [m]$. Donc, 
\begin{eqnarray*}
M &=& M_{l'}(x+m\mathbb{Z},k+m\mathbb{Z},\ldots,k+m\mathbb{Z},x+m\mathbb{Z}) \\
  &=& M_{l'}(0+m\mathbb{Z},-2+m\mathbb{Z}\ldots,-2+m\mathbb{Z},0+m\mathbb{Z}) \\
  &=& (-1+m\mathbb{Z})(-1+m\mathbb{Z})^{w}Id \\
	&=& (-1+m\mathbb{Z})^{w+1}Id \\
	&=& (1+m\mathbb{Z})Id~~({\rm car}~w~{\rm est~impair}).
\end{eqnarray*}

\end{itemize}

\noindent Par le lemme chinois ($m$ et $n$ premiers entre eux), $M_{l'}(\overline{x},\overline{k},\ldots,\overline{k},\overline{x})=Id$. De plus, \[l'=2+mw \leq 2+m(2n-1)=2mn-(m-2) \leq 2mn-1.\] 
\noindent Ainsi, la solution $\overline{k}$-monomiale minimale de \eqref{p} est réductible.
\\
\\Supposons par l'absurde que $k$ et $N$ ne soient pas premiers entre eux. Il existe un nombre premier $p$ qui divise $k$ et $N=nm$. Par le lemme d'Euclide, $p$ divise $n$ ou $p$ divise $m$. On peut supposer, sans perte de généralité, que $p$ divise $n$. Comme $m$ et $n$ sont premiers entre eux, $p$ ne divise pas $m$. Ainsi, $p$ divise $k+2bn=2am$. Comme $p$ ne divise pas $2m$, $p$ divise $a$, par le lemme de Gauss. Donc, $p$ divise $am+bn=1$, ce qui est absurde. Ainsi, $k$ est premier avec $N$.

\end{proof}

\begin{examples}
{\rm Reprenons les deux exemples considérés avec la proposition \ref{36}.
\begin{itemize}
\item $N=15=3 \times 5$. On a $n=3$ et $m=5$. On a $5\times (-1) +3 \times 2=1$ et on pose $a=-1$ et $b=2$. Posons $k=2 \times 5 \times (-1)-2 \times 3 \times 2= -22 \equiv 8 [15]$. $5=3+2$, $2 \times 5 +2 \equiv 0 [3]$, c'est-à-dire $r=2$, $w=5$ et $x=2 \times 5 \times (-1) \equiv 5 [15]$. Par le résultat précédent, la solution $\overline{k}$-monomiale minimale de \eqref{p} est réductible. Elle est de taille 30 et on peut la réduire avec la solution $(\overline{5},\overline{k},\ldots,\overline{k},\overline{5})$ de taille $l'=2+5 \times 5=27$. Notons que l'on peut également la réduire avec la solution $(\overline{3},\overline{k},\overline{k},\overline{k},\overline{3})$, puisque $(\overline{k},\ldots,\overline{k})=(\overline{3},\overline{k},\overline{k},\overline{k},\overline{3}) \oplus (\overline{5},\overline{k},\ldots,\overline{k},\overline{5})$.
\\

\item $N=17~441=107 \times 163$. On a $n=107$ et $m=163$. On a $107 \times 32 +163 \times (-21)=3424-3423=1$ et on pose $a=-21$ et $b=32$. Posons $k=2 \times 163 \times (-21)-2 \times 107 \times 32=-6846-6848=-13~694 \equiv 3747 [17~441]$. On a $163=107+56$ et $56 \times 149 +2 \equiv 0 [107]$, c'est-à-dire $r=56$, $w=149$ et $x=2 \times 163 \times (-21)=-6846$. Par le résultat précédent, la solution $\overline{k}$-monomiale minimale de \eqref{p} est réductible. Elle est de taille 34~882 et on peut la réduire avec la solution $(\overline{-6846},\overline{k},\ldots,\overline{k},\overline{-6846})$ de taille $l'=2+163 \times 149=24~289$.

\end{itemize}
}
\end{examples}

\noindent Grâce à cette proposition et au théorème \ref{33}, on a redémontré le point i) du théorème \ref{211}.

\subsection{Le cas des entiers pairs}

On s'intéresse maintenant à l'étude des entiers semi monomialement irréductibles pairs. Grâce aux résultats déjà établis, on a :

\begin{proposition}
\label{52}

Les entiers de la forme $2p^{n}$, avec $p$ premier et $n \geq 1$, sont semi monomialement irréductibles.

\end{proposition}

\begin{proof}

Si $p=2$, c'est une conséquence du théorème \ref{33}. Si $p$ est impair, la démonstration est identique à la preuve du cas \og $k=2a$ avec $a$ non divisible par 3 \fg étudié pour établir le théorème \ref{28}.

\end{proof}

\begin{examples}
{\rm $34=2 \times 17$, $2662=2\times 11^{3}$ et $24~334=2\times 23^{3}$ sont semi monomialement irréductibles. 
}

\end{examples}

Il reste à considérer les entiers de la forme $N=2^{a}3^{b}5^{c}7^{d}17^{e}31^{f}127^{g}$ et $N=2\times 3^{b}5^{c}$. Pour cela, on commence par le lemme suivant :

\begin{lemma}
\label{520}

Soient $p$ un nombre premier, $k$ un entier et $n \in \mathbb{N}^{*}$. Soit $r$ la taille de la solution $\overline{k}$-monomiale minimale de $(E_{p^{n}})$. 
\\
\\i) La taille de la solution $\overline{k}$-monomiale minimale de $(E_{p^{n+1}})$ est égale à $r$ ou à $pr$. 
\\
\\ii) Si $p$ est impair et si $r$ est pair alors $M_{r}(\overline{k},\ldots,\overline{k})=-Id$.
\\
\\iii) Soit $h$ la taille de la solution $(k+p\mathbb{Z})$-monomiale minimale de $(E_{p})$. Si $p$ est impair alors $r \equiv \pm h [4]$. En particulier, $r \equiv 2 [4]$ si et seulement si $h \equiv 2 [4]$.

\end{lemma}

\begin{proof}

i) Il existe $A \in {\rm M}_{2}(\mathbb{Z})$ et $\epsilon \in \{\pm 1 \}$ tel que $M_{r}(k,\ldots,k)=\epsilon Id +p^{n}A$. On a par le binôme de Newton :

\[M_{pr}(k,\ldots,k)=(\epsilon Id +p^{n}A)^{p}=\sum_{i=0}^{p} {p \choose i} \epsilon^{p-i}p^{ni}A^{i}.\]

\noindent Comme $p^{n+1}$ divise ${p \choose i} \epsilon^{p-i}p^{ni}$ pour $i \geq 1$, on a \[M_{pr}(k+p^{n+1}\mathbb{Z},\ldots,k+p^{n+1}\mathbb{Z})=(\epsilon^{p}+p^{n+1}\mathbb{Z})Id.\]

\noindent Ainsi, la taille $l$ de la solution $\overline{k}$-monomiale minimale de $(E_{p^{n+1}})$ divise $pr$. Or, $l$ est un multiple de $r$ (lemme \ref{30}) et $p$ est premier. Donc, $l=r$ ou $l=pr$. 
\\
\\ ii) On raisonne par récurrence sur $n$.
\\
\\Soit $h$ la taille de la solution $(k+p\mathbb{Z})$-monomiale minimale de $(E_{p})$, qu'on suppose paire.
\\
\\Soit $A=M_{1}(k+p\mathbb{Z})$. Par définition de $h$, $A^{h}=\pm Id$. Si $A^{h}=Id$ alors $(A^{\frac{h}{2}}-Id)(A^{\frac{h}{2}}+Id)=0$. Donc, $(X-1)(X+1)$ est un polynôme annulateur scindé à racines simples sur $\mathbb{Z}/p\mathbb{Z}$ de $B=A^{\frac{h}{2}}$. Ainsi, $B$ est diagonalisable et ses valeurs propres sont égales à $\pm 1+p\mathbb{Z}$ ($\mathbb{Z}/p\mathbb{Z}$ est un corps puisque $p$ est premier). Comme $B$ est de déterminant $1+p\mathbb{Z}$, $B$ est semblable à $\pm Id$, donc égale à $\pm Id$, ce qui contredit la minimalité de $h$. Ainsi, $A^{h}=-Id$.
\\
\\Supposons qu'il existe un $n \in \mathbb{N}^{*}$ tel que les solutions de taille paire de $(E_{p^{n}})$ vérifient la propriété souhaitée. Soit $k+p^{n+1}\mathbb{Z}$ tel que la solution $(k+p^{n+1}\mathbb{Z})$-monomiale minimale de $(E_{p^{n+1}})$ est de taille paire (on peut toujours trouver un tel $k$ puisqu'on peut par exemple prendre $k=0$). Notons $l$ cette taille. Il existe $\epsilon \in \{-1, 1\}$ tel que $M_{l}(k+p^{n+1}\mathbb{Z},\ldots,k+p^{n+1}\mathbb{Z})=(\epsilon+p^{n+1}\mathbb{Z})Id$.
\\
\\Soit $r$ la taille de la solution $(k+p^{n}\mathbb{Z})$-monomiale minimale de $(E_{p^{n}})$. Par i), $l=r$ ou $l=pr$. Comme $l$ est pair et $p$ impair, $r$ est pair. Par hypothèse de récurrence, $M_{r}(k+p^{n}\mathbb{Z},\ldots,k+p^{n}\mathbb{Z})=(-1+p^{n}\mathbb{Z})Id$. On distingue deux cas :
\begin{itemize}
\item $l=r$. Si $\epsilon=1$ alors $M_{r}(k+p^{n+1}\mathbb{Z},\ldots,k+p^{n+1}\mathbb{Z})=(1+p^{n+1}\mathbb{Z})Id$. En particulier, on a $M_{r}(k+p^{n}\mathbb{Z},\ldots,k+p^{n}\mathbb{Z})=(1+p^{n}\mathbb{Z})Id$. Or, $1+p^{n}\mathbb{Z} \neq -1+p^{n}\mathbb{Z}$, puisque $p^{n} \geq 3$. Donc, $\epsilon=-1$.
\\
\item $l=pr$. Si $\epsilon=1$ alors $M_{pr}(k+p^{n+1}\mathbb{Z},\ldots,k+p^{n+1}\mathbb{Z})=(1+p^{n+1}\mathbb{Z})Id$. En particulier, on a $M_{pr}(k+p^{n}\mathbb{Z},\ldots,k+p^{n}\mathbb{Z})=(1+p^{n}\mathbb{Z})Id$. Or, 
\begin{eqnarray*}
M_{pr}(k+p^{n}\mathbb{Z},\ldots,k+p^{n}\mathbb{Z}) &=& M_{r}(k+p^{n}\mathbb{Z},\ldots,k+p^{n}\mathbb{Z})^{p} \\
                                                   &=& (-1+p^{n}\mathbb{Z})^{p}Id \\
																									 &=& (-1+p^{n}\mathbb{Z})Id~~~~(p~{\rm impair}).
\end{eqnarray*}																									
\noindent De plus, $1+p^{n}\mathbb{Z} \neq -1+p^{n}\mathbb{Z}$, puisque $p^{n} \geq 3$. Donc, $\epsilon=-1$.
\\
\end{itemize} 

\noindent Par récurrence, le résultat est démontré.
\\
\\iii) Soit $h$ la taille de la solution $(k+p\mathbb{Z})$-monomiale minimale de $(E_{p})$.
\\
\\Montrons par récurrence sur $n$ que la taille de la solution $(k+p^{n}\mathbb{Z})$-monomiale minimale de $(E_{p^{n}})$ est de la forme $hp^{a}$ avec $a \in \mathbb{N}$. Si $n=1$ le résultat est vrai (on prend $a=0$). Supposons qu'il existe un $n \in \mathbb{N}^{*}$ tel que la taille $r$ de la solution $(k+p^{n}\mathbb{Z})$-monomiale minimale de $(E_{p^{n}})$ est de la forme $hp^{a}$ avec $a \in \mathbb{N}$. Par i), la taille de la solution $(k+p^{n+1}\mathbb{Z})$-monomiale minimale de $(E_{p^{n+1}})$ est égale à $r=hp^{a}$ ou à $pr=hp^{a+1}$. Par récurrence, le résultat est démontré.
\\
\\Comme $p$ est impair, on a, pour tout $a \in \mathbb{N}$, $p^{a} \equiv \pm 1 [4]$. Ainsi, pour tout $a \in \mathbb{N}$, $hp^{a} \equiv \pm h [4]$ et donc $hp^{a} \equiv 2 [4]$ si et seulement si $h \equiv 2 [4]$.
 
\end{proof}

\begin{remark}
{\rm
Le résultat de ii) n'est plus vrai si $p=2$. Par exemple, si $N=4$ alors la solution $\overline{2}$-monomiale minimale de \eqref{p} est de taille 4 avec $M_{4}(\overline{2},\overline{2},\overline{2},\overline{2})=Id$. De plus, on ne peut pas obtenir un résultat similaire à ii) pour les solutions monomiales minimales de taille impaire puisqu'on a la formule : \[M_{r}(\overline{k},\ldots,\overline{k})=\overline{(-1)^{r}}M_{r}(\overline{-k},\ldots,\overline{-k})^{T}.\]
} 
\end{remark}

\noindent On peut maintenant établir le résultat général ci-dessous :

\begin{proposition}
\label{53}

Soient $p \in \mathbb{P}$ impair, $a$ un entier impair non divisible par $p$, $m \geq 1$ et $N=4 \times p^{m}$. Notons $h$ la taille de la solution $(2a+p \mathbb{Z})$-monomiale minimale de $(E_{p})$. 
\begin{itemize}
\item Si $h \equiv 2 [4]$
alors la solution $\overline{2a}$-monomiale minimale de \eqref{p} est réductible.
\item Si $h \equiv -1, 0, 1 [4]$ alors la solution $\overline{2a}$-monomiale minimale de \eqref{p} est irréductible.
\end{itemize}

\noindent En particulier, $N$ est semi monomialement irréductible si et seulement si pour tout $a$ impair non divisible par $p$ la taille de la solution $(2a+p \mathbb{Z})$-monomiale minimale de $(E_{p})$ n'est pas congrue à 2 modulo 4.

\end{proposition}

\begin{proof}

Notons $r$ la taille de la solution $\overline{2a}$-monomiale minimale de \eqref{p} et $l$ la taille de la solution $(2a+p^{m} \mathbb{Z})$-monomiale minimale de $(E_{p^{m}})$. Par le lemme \ref{520}, $l \equiv \pm h [4]$. 
\\
\\Comme $p^{m}$ divise $N$, $r$ est un multiple de $l$. De plus, la solution $(2a+4 \mathbb{Z})$-monomiale minimale de $(E_{4})$ est de taille 4 avec $M_{4}(2a+4 \mathbb{Z},\ldots,2a+4 \mathbb{Z})=(1+4\mathbb{Z}) Id$. Ainsi, $r$ est un multiple de ${\rm ppcm}(4,l)$.
\\
\\Il existe $\epsilon \in \{1, -1\}$ tel que $M_{l}(2a+p^{m} \mathbb{Z},\ldots,2a+p^{m} \mathbb{Z})=(\epsilon +p^{m}\mathbb{Z}) Id$. On considère les deux valeurs possibles de $\epsilon$ :
\\
\\ \underline{i) On suppose $\epsilon=1$ :} Par le lemme \ref{520}, $l$ est impair et donc $l \equiv \pm 1 [4]$.
\\
\\On a ${\rm ppcm}(4,l)=4l$ et $M_{4l}(2a,\ldots,2a)=Id$ modulo $p^{m}$ et modulo 4. Par le lemme chinois, $r=4l$. Supposons par l'absurde que la solution $\overline{2a}$-monomiale minimale de \eqref{p} est réductible. Il existe une solution $(\overline{x},\overline{2a},\ldots,\overline{2a},\overline{x})$ de \eqref{p} de taille $3 \leq l' \leq 2l+1$. 
\\
\\Comme $p^{m}$ divise $N$, le $l'$-uplet $(x+p^{m} \mathbb{Z},2a+p^{m} \mathbb{Z},\ldots,2a+p^{m} \mathbb{Z},x+p^{m} \mathbb{Z})$ est une solution de $(E_{p^{m}})$. Comme la solution $(2a+p^{m} \mathbb{Z})$-monomiale minimale de $(E_{p^{m}})$ est irréductible de taille $l$, on a nécessairement $l' \in \{l, l+2, 2l\}$ (voir lemme \ref{31}). 
\begin{itemize}
\item Si $l'=l$. Comme $4$ divise $N$, le $l$-uplet $(x+4 \mathbb{Z},2a+4 \mathbb{Z},\ldots,2a+4 \mathbb{Z},x+4 \mathbb{Z})$ est une solution de $(E_{4})$ ce qui est impossible puisque la solution $(2a+4 \mathbb{Z})$-monomiale minimale de $(E_{4})$ est irréductible de taille $4$ et $l$ est impair (voir lemme \ref{31}).
\item Si $l'=l+2$ alors on peut effectuer le même raisonnement que dans le cas précédent.
\item Si $l'=2l$. On a $l'\equiv 2 [4]$ et, par le lemme \ref{31}, $x+p^{m} \mathbb{Z}=2a+p^{m} \mathbb{Z}$ et $x+4 \mathbb{Z}=0+4 \mathbb{Z}$. Ainsi, 
\[M_{l'}(x+p^{m} \mathbb{Z},2a+p^{m} \mathbb{Z},\ldots,2a+p^{m} \mathbb{Z},x+p^{m} \mathbb{Z})=(\epsilon+p^{m}  \mathbb{Z})^{2}Id=(1+p^{m}  \mathbb{Z})Id\]
 et 
\[M_{l'}(x+4 \mathbb{Z},2a+4 \mathbb{Z},\ldots,2a+4 \mathbb{Z},x+4 \mathbb{Z})=(-1+4  \mathbb{Z})Id.\] 
\noindent Ceci est absurde.

\end{itemize}

\noindent \underline{ii) On suppose $\epsilon=-1$ :} On considère les différentes valeurs de $l$ modulo 4.
\\
\\ \uwave{On suppose $l \equiv \pm 1 [4]$.} On a ${\rm ppcm}(4,l)=4l$ et on a $M_{4l}(2a,\ldots,2a)=Id$ modulo $p^{m}$ et modulo 4. On peut effectuer le même raisonnement que dans le cas $\epsilon=1$.
\\
\\\uwave{On suppose $l \equiv 0 [4]$.} On a ${\rm ppcm}(4,l)=l$. Toutefois, comme $\epsilon=-1$, $r \neq l$. En revanche, on a $M_{2l}(2a,\ldots,2a)= Id$ modulo $p^{m}$ et modulo 4. Par le lemme chinois, $r=2l$. Supposons par l'absurde que la solution $\overline{k}$-monomiale minimale de \eqref{p} est réductible. Il existe une solution $(\overline{x},\overline{2a},\ldots,\overline{2a},\overline{x})$ de \eqref{p} de taille $3 \leq l' \leq l+1$. 
\\
\\Comme $p^{m}$ divise $N$, le $l'$-uplet $(x+p^{m} \mathbb{Z},2a+p^{m} \mathbb{Z},\ldots,2a+p^{m} \mathbb{Z},x+p^{m} \mathbb{Z})$ est une solution de $(E_{p^{m}})$. Comme la solution $(2a+p^{m} \mathbb{Z})$-monomiale minimale de $(E_{p^{m}})$ est irréductible de taille $l$, on a nécessairement $l'=l$ (voir lemme \ref{31}). Comme $l'\equiv 0 [4]$, on a, par le lemme \ref{31}, $x+p^{m} \mathbb{Z}=2a+p^{m} \mathbb{Z}$ et $x+4 \mathbb{Z}=2a+4 \mathbb{Z}$ . Ainsi, $M_{l'}(x+p^{m} \mathbb{Z},2a+p^{m} \mathbb{Z},\ldots,2a+p^{m} \mathbb{Z},x+p^{m} \mathbb{Z})=(-1+p^{m}  \mathbb{Z})Id$ et $M_{l}(x+4 \mathbb{Z},2a+4 \mathbb{Z},\ldots,2a+4 \mathbb{Z},x+4 \mathbb{Z})=(1+4  \mathbb{Z})Id$. Ceci est absurde.
\\
\\ \uwave{On suppose $l \equiv 2 [4]$.} On a ${\rm ppcm}(4,l)=2l$ et on a $M_{2l}(2a,\ldots,2a)= Id$ modulo $p^{m}$ et modulo 4. Ainsi, par le lemme chinois, $r=2l$. Soit $x$ tel que $x+p^{m} \mathbb{Z}=2a+p^{m} \mathbb{Z}$ et $x+4 \mathbb{Z}=0+4 \mathbb{Z}$. On a \[M_{l}(x+p^{m} \mathbb{Z},2a+p^{m} \mathbb{Z},\ldots,2a+p^{m} \mathbb{Z},x+p^{m} \mathbb{Z})=(-1+p^{m}  \mathbb{Z})Id\] 
et 
\[M_{l}(x+4 \mathbb{Z},2a+4 \mathbb{Z},\ldots,2a+4 \mathbb{Z},x+4 \mathbb{Z})=(-1+4  \mathbb{Z})Id.\]
\noindent Par le lemme chinois, $M_{l}(\overline{x},\overline{2a},\ldots,\overline{2a},\overline{x})=-Id$ et donc on peut utiliser le $l$-uplet $(\overline{x},\overline{2a},\ldots,\overline{2a},\overline{x})$ pour réduire la solution $\overline{2a}$-monomiale minimale de \eqref{p}.

\end{proof}

\begin{example}
{\rm 
Soit $N=44=4 \times 11$. La solution $(6+11\mathbb{Z})$-monomiale minimale de $(E_{11})$ est de taille $6 \equiv 2 [4]$. Par la proposition précédente, 44 est semi monomialemnt réductible.
}
\end{example}

\noindent On dispose également du résultat général suivant :

\begin{theorem}
\label{54}

Soient $p \in \mathbb{P}$ impair, $l \geq 0$, $m \geq 1$ impair et premier avec $p$, $n \geq 2$, $N=2^{n}\times m \times p^{l}$ et $a$ un entier premier avec $N$. On suppose que la solution $2a$-monomiale minimale de $(E_{2^{n}m})$ est irréductible. On suppose qu'il existe $2 \leq b \leq n$ tel que la solution $2a$-monomiale minimale de $(E_{2^{b}p})$ est irréductible. La solution $\overline{2a}$-monomiale minimale de \eqref{p} est irréductible. 
\\
\\En particulier, s'il existe $2 \leq b \leq n$ tel que $2^{n}m$ et $2^{b}p$ sont semi monomialement irréductibles alors $N$ l'est aussi.

\end{theorem}

\begin{proof}

Si $l=0$ alors le résultat est vrai par hypothèse. On suppose donc $l \geq 1$. Soit $k=2a$ avec $a$ un entier premier avec $N$. Par le théorème de Bézout, il existe $(u,v) \in \mathbb{Z}^{2}$ tels que $u2^{n}m+vp^{l}=1$. En particulier, $u$ n'est pas divisible par $p$ et $v$ est premier avec $2^{n}m$.
\\
\\Notons $r$ la taille de la solution $\overline{k}$-monomiale minimale de \eqref{p}. Supposons par l'absurde que cette dernière est réductible. Il existe une solution $(\overline{x},\overline{k},\ldots,\overline{k},\overline{x})$ de \eqref{p} de taille $3 \leq d \leq r-1$ permettant de réduire. On va considérer plusieurs diviseurs de $N$.

\begin{itemize}
\item $2^{n}m$ divise $N$. Le $d$-uplet $(x+2^{n}m \mathbb{Z},k+2^{n}m \mathbb{Z},\ldots,k+2^{n}m \mathbb{Z},x+2^{n}m \mathbb{Z})$ est une solution de $(E_{2^{n}m})$. Or, par hypothèse, la solution $(k+2^{n}m \mathbb{Z})$-monomiale minimale de $(E_{2^{n}m})$ est irréductible. Donc, par le lemme \ref{31}, $x \equiv 0, k [2^{n}m]$.
\\
\item $p^{l}$ divise $N$. Le $d$-uplet $(x+p^{l} \mathbb{Z},k+p^{l} \mathbb{Z},\ldots,k+p^{l} \mathbb{Z},x+p^{l} \mathbb{Z})$ est une solution de $(E_{p^{l}})$. Or, comme $k$ n'est pas divisible par $p$, la solution $(k+p^{l} \mathbb{Z})$-monomiale minimale de $(E_{p^{l}})$ est irréductible (Théorème \ref{33}). Par le lemme \ref{31}, $x \equiv 0, k [p^{l}]$.
\\
\item $h=2^{b}p$ divise $N$. Le $d$-uplet $(x+h \mathbb{Z},k+h \mathbb{Z},\ldots,k+h \mathbb{Z},x+h \mathbb{Z})$ est une solution de $(E_{h})$. Or, par hypothèse, la solution $(k+h \mathbb{Z})$-monomiale minimale de $(E_{h})$ est irréductible. Par le lemme \ref{31}, $x \equiv 0, k [h]$.
\\
\end{itemize}

\noindent Si $x \equiv 0 [2^{n}m]$ et $x \equiv 0 [p^{l}]$ alors, par le lemme chinois, $\overline{x}=\overline{0}$. Dans ce cas, le $d-2$-uplet $(\overline{k},\ldots,\overline{k})$ est une solution de \eqref{p}, ce qui est absurde (puisque $d-2 < r$). Si $x \equiv k [2^{n}m]$ et $x \equiv k [p^{l}]$ alors, par le lemme chinois, $\overline{x}=\overline{k}$. Dans ce cas, le $d$-uplet $(\overline{k},\ldots,\overline{k})$ est une solution de \eqref{p}, ce qui est absurde (puisque $d < r$). Donc, on a $x \equiv 0 [2^{n}m]$ et $x \equiv k [p^{l}]$ ou $x \equiv k [2^{n}m]$ et $x \equiv 0 [p^{l}]$.
\\
\\On suppose $x \equiv 0 [2^{n}m]$ et $x \equiv k [p^{l}]$. Par le lemme chinois, on a $x \equiv uk2^{n}m [N]$. Si $x \equiv 0 [h]$ alors $p$ divise $uk2^{n}m$, ce qui est absurde. Donc, $x \equiv k [h]$ et 4 divise $k(u2^{n}m-1)$ (puisque $b \geq 2$). Or, ceci est absurde car $(u2^{n}m-1)$ est impair et $k$ n'est pas divisible par 4.
\\
\\On suppose $x \equiv k [2^{n}m]$ et $x \equiv 0 [p^{l}]$. Par le lemme chinois, on a $x \equiv vkp^{l} [N]$. Si $x \equiv 0 [h]$ alors 4 divise $vkp^{l}$, ce qui est absurde. Donc, $x \equiv k [h]$ et $p$ divise $k(vp^{l}-1)$. Or, ceci est absurde car $k$ et $(vp^{l}-1)$ ne sont pas divisibles par $p$.
\\
\\Ainsi, la solution $\overline{k}$-monomiale minimale de \eqref{p} est irréductible.

\end{proof}

\noindent Avant d'appliquer les deux résultats que nous venons de démontrer, on a besoin du lemme ci-dessous :

\begin{lemma}
\label{55}

Soient $p \in \{3, 5, 7, 17, 31, 127\}$ et $k$ un entier non divisible par $p$. On note $l$ la taille de la solution $\overline{k}$-monomiale minimale de $(E_{p})$. On a $l \not\equiv 2 [4]$.

\end{lemma}

\begin{proof}

i) Si $p=3$ alors $k+3\mathbb{Z}= \pm 1+3\mathbb{Z}$ et la taille de la solution $(k+3\mathbb{Z})$-monomiale minimale de $(E_{3})$ est égale à 3.
\\
\\ii) Si $p=5$ alors $k+5\mathbb{Z} \in \{\pm 1+5\mathbb{Z}, \pm 2+5\mathbb{Z}\}$. Ainsi, la taille de la solution $(k+5\mathbb{Z})$-monomiale minimale de $(E_{5})$ est égale à 3 ou 5.
\\
\\iii) Si $p=7$. Si $k+7\mathbb{Z}=\pm 1+7\mathbb{Z}$ alors la taille de la solution $(k+7\mathbb{Z})$-monomiale minimale de $(E_{7})$ est égale à 3. Si $k+7\mathbb{Z}=\pm 2+7\mathbb{Z}$ alors la taille de la solution $(k+7\mathbb{Z})$-monomiale minimale de $(E_{7})$ est égale à 7. Si $k+7\mathbb{Z}=\pm 3+7\mathbb{Z}$ alors la taille de la solution $(k+7\mathbb{Z})$-monomiale minimale de $(E_{7})$ est égale à 4. En particulier, on obtient dans tous les cas $l \not \equiv 2 [4]$.
\\
\\iv) Pour $p=17$, on utilise les tailles des solutions monomiales minimales de $(E_{17})$ regroupées ci-dessous et le fait que les solutions $\pm \overline{k}$-monomiales minimales ont la même taille :

\begin{center}
\begin{tabular}{|c|c|c|c|c|c|c|c|c|}
\hline
   $\overline{k}$  & $\overline{1}$ & $\overline{2}$ & $\overline{3}$ & $\overline{4}$ & $\overline{5}$ & $\overline{6}$ & $\overline{7}$ & $\overline{8}$    \rule[-7pt]{0pt}{18pt} \\
	\hline
	{\rm taille}  & 3 & 17 & 9 & 9 & 8 & 4 & 9 & 8    \rule[-7pt]{0pt}{18pt} \\
	\hline

\end{tabular}
\end{center}	

\noindent v) Pour $p=31$, on obtient de même :

\begin{center}
\begin{tabular}{|c|c|c|c|c|c|c|c|c|c|c|c|c|c|c|c|}
\hline
   $\overline{k}$  & $\overline{1}$ & $\overline{2}$ & $\overline{3}$ & $\overline{4}$ & $\overline{5}$ & $\overline{6}$ & $\overline{7}$ & $\overline{8}$ & $\overline{9}$ & $\overline{10}$ & $\overline{11}$ & $\overline{12}$ & $\overline{13}$ & $\overline{14}$ & $\overline{15}$     \rule[-7pt]{0pt}{18pt} \\
	\hline
	{\rm taille}  & 3 & 31 & 15 & 16 & 8 & 15 & 15 & 4 & 16 & 16 & 16 & 5 & 5 & 8 & 15   \rule[-7pt]{0pt}{18pt} \\
	\hline

\end{tabular}
\end{center}

\noindent vi) Pour $p=127$, on procède de façon analogue :

\begin{center}
\begin{tabular}{|c|c|c|c|c|c|c|c|c|c|c|c|c|c|c|c|c|c|c|c|c|c|}
\hline
   $\overline{k}$  & $\overline{1}$ & $\overline{2}$ & $\overline{3}$ & $\overline{4}$ & $\overline{5}$ & $\overline{6}$ & $\overline{7}$ & $\overline{8}$ & $\overline{9}$ & $\overline{10}$ & $\overline{11}$ & $\overline{12}$ & $\overline{13}$ & $\overline{14}$ & $\overline{15}$ & $\overline{16}$ & $\overline{17}$ & $\overline{18}$ & $\overline{19}$ & $\overline{20}$ & $\overline{21}$  \rule[-7pt]{0pt}{18pt} \\
	\hline
	{\rm taille}  & 3 & 127 & 64 & 64 & 63 & 63 & 32 & 63 & 16 & 64 & 63 & 63 & 9 & 32 & 63 & 4 & 21 & 64 & 63 & 21& 64    \rule[-7pt]{0pt}{18pt} \\
	\hline

\end{tabular}
\end{center}

\begin{center}
\begin{tabular}{|c|c|c|c|c|c|c|c|c|c|c|c|c|c|c|c|c|c|c|c|c|c|}
\hline
   $\overline{k}$  & $\overline{22}$ & $\overline{23}$ & $\overline{24}$ & $\overline{25}$ & $\overline{26}$ & $\overline{27}$ & $\overline{28}$ & $\overline{29}$ & $\overline{30}$ & $\overline{31}$ & $\overline{32}$ & $\overline{33}$ & $\overline{34}$ & $\overline{35}$ & $\overline{36}$ & $\overline{37}$ & $\overline{38}$ & $\overline{39}$ & $\overline{40}$ & $\overline{41}$ & $\overline{42}$  \rule[-7pt]{0pt}{18pt} \\
	\hline
	{\rm taille}  & 63 & 63 & 7 & 63 & 63 & 64 & 21 & 32 & 32 & 63 & 63 & 21 & 63 & 32 & 7 & 64 & 64 & 63 & 9 & 21 & 8    \rule[-7pt]{0pt}{18pt} \\
	\hline

\end{tabular}
\end{center}

\begin{center}
\begin{tabular}{|c|c|c|c|c|c|c|c|c|c|c|c|c|c|c|c|c|c|c|c|c|c|}
\hline
   $\overline{k}$  & $\overline{43}$ & $\overline{44}$ & $\overline{45}$ & $\overline{46}$ & $\overline{47}$ & $\overline{48}$ & $\overline{49}$ & $\overline{50}$ & $\overline{51}$ & $\overline{52}$ & $\overline{53}$ & $\overline{54}$ & $\overline{55}$ & $\overline{56}$ & $\overline{57}$ & $\overline{58}$ & $\overline{59}$ & $\overline{60}$ & $\overline{61}$ & $\overline{62}$ & $\overline{63}$  \rule[-7pt]{0pt}{18pt} \\
	\hline
	{\rm taille}  & 64 & 64 & 32 & 64 & 16 & 8 & 64 & 16 & 64 & 64 & 9 & 64 & 63 & 21 & 63 & 32 & 32 & 16 & 7 & 63 & 64    \rule[-7pt]{0pt}{18pt} \\
	\hline

\end{tabular}
\end{center}

\end{proof}

\noindent En combinant la proposition \ref{53}, le théorème \ref{54} et le lemme \ref{55}, on obtient le résultat ci-dessous :

\begin{corollary}
\label{56}

Les entiers de la forme $2^{a}3^{b}5^{c}7^{d}17^{e}31^{f}127^{g}$ avec $a \geq 2$ et $b,c,d,e,f,g \geq 0$ sont semi monomialement irréductibles.

\end{corollary}

\begin{proof}

Par la proposition \ref{53} et le lemme \ref{55}, les entiers de la forme $4p$ avec $p \in \{3, 5, 7, 17, 31, 127\}$ sont semi monomialement irréductibles.
\\
\\Soient $a \geq 2$ et $p \in \{3, 5, 7, 17, 31, 127\}$. Soit $b \geq 0$. $2^{a} \times 1$ est semi monomialement irréductible et $4p$ est semi monomialement irréductible. Par le théorème \ref{54}, $2^{a}p^{b}$ est semi monomialement irréductible. 
\\
\\Soient $q \in \{3, 5, 7, 17, 31, 127\}-\{p\}$ et $c \geq 0$. $2^{a}p^{b}$ est semi monomialement irréductible, par ce qui précède, et $4q$ est semi monomialement irréductible. Par le théorème \ref{54}, $2^{a}p^{b}q^{c}$ est semi monomialement irréductible. 
\\
\\En continuant ainsi, on obtient le résultat souhaité.

\end{proof}

\begin{examples}
{\rm $84=2^{2} \times 3 \times 7$ et $680=2^{3} \times 5 \times 17$ sont semi monomialement irréductibles.
}
\end{examples}

Les résultats qui viennent d'être établis montrent l'intérêt des nombres premiers pour lesquels les solutions monomiales minimales non nulles ont une taille qui n'est pas congrue à 2 modulo 4. Cela soulève naturellement la question de savoir si ces derniers sont nombreux. Informatiquement, on a cherché tous les nombres premiers inférieurs à $60~000$ vérifiant cette propriété. On a obtenu la liste très restreinte ci-dessous : \[\{3, 5, 7, 17, 31, 127, 257, 8191\}.\]

\noindent Cela invite donc à formuler le problème suivant : 

\begin{pro}

Trouver l'ensemble des nombres premiers pour lesquels toutes les solutions monomiales minimales non nulles ont une taille qui n'est pas congrue à 2 modulo 4.

\end{pro}

Si $a=1$ alors le résultat du corollaire \ref{56} n'est plus forcément vrai. Par exemple, considérons le cas $N=42=2 \times 3 \times 7$. La solution $\overline{10}$-monomiale minimale de \eqref{p} est de taille 24. On peut la réduire avec la solution $(\overline{28},\overline{10},\overline{10},\overline{10},\overline{10},\overline{28})$. Toutefois, on dispose tout de même du résultat suivant :

\begin{proposition}
\label{57}

Les entiers de la forme $2\times 3^{n}5^{m}$ avec $n,m \geq 0$ et $n+m \neq 0$ sont semi monomialement irréductibles.

\end{proposition}

\begin{proof}

Si $n=0$ ou $m=0$ alors le résultat est vrai (proposition \ref{52}). 
\\
\\Si $N=2 \times 3 \times 5$. Pour montrer que 30 est semi monomialement irréductible, il faut considérer les solutions $\overline{k}$-monomiales minimales avec $\overline{k} \in \{\pm \overline{2}, \pm \overline{4}, \pm \overline{8}, \pm \overline{14}\}$. En utilisant les résultats de la section \ref{pre}, il suffit de s'intéresser aux cas $\overline{k} \in \{\overline{4}, \overline{8}, \overline{14}\}$ : 
\begin{itemize}
\item La solution $\overline{4}$-monomiale minimale de \eqref{p} est de taille 6. Si cette dernière est réductible alors elle est la somme de deux solutions de taille 4. Or, il n'y a pas de solution de la forme $(\overline{a},\overline{4},\overline{4},\overline{a})$ car $\overline{4}^{2} \neq \overline{0}, \overline{2}$.
\\
\item La solution $\overline{8}$-monomiale minimale de \eqref{p} est de taille 30. Les racines modulo 30 du polynôme $X(X-\overline{8})$ sont : $\overline{0}, \overline{8}, \overline{18}, \overline{20}$. On vérifie, par un calcul direct, qu'il n'y a pas de solution de la forme $(\overline{18},\overline{8},\ldots,\overline{8},\overline{18})$ ou $(\overline{20},\overline{8},\ldots,\overline{8},\overline{20})$ de taille inférieure à 29.
\\
\item La solution $\overline{14}$-monomiale minimale de \eqref{p} est de taille 6. Si cette dernière est réductible alors elle est la somme de deux solutions de taille 4. Or, il n'y a pas de solution de la forme $(\overline{a},\overline{14},\overline{14},\overline{a})$ car $\overline{14}^{2} \neq \overline{0}, \overline{2}$.
\\
\end{itemize}

\noindent Ainsi, 30 est semi monomialement irréductible.
\\
\\Si $N=2\times 3^{n}5^{m}$, $n, m \geq 1$ et $n+m \geq 3$. Soit $a$ un entier premier avec $\frac{N}{2}$ et $k=2a$. Par le théorème de Bézout, il existe $(u,v) \in \mathbb{Z}^{2}$ tels que $2u\times 3^{n}+v5^{m}=1$. En particulier, $u$ n'est pas divisible par $5$ et $v$ est premier avec $2 \times 3^{n}$.
\\
\\Notons $r$ la taille de la solution $\overline{k}$-monomiale minimale de \eqref{p}. Supposons par l'absurde que cette dernière est réductible. Il existe une solution $(\overline{x},\overline{k},\ldots,\overline{k},\overline{x})$ de \eqref{p} de taille $3 \leq d \leq r-1$ permettant de réduire. On va considérer plusieurs diviseurs de $N$.

\begin{itemize}
\item $h=2\times 3^{n}$ divise $N$. Le $d$-uplet $(x+h \mathbb{Z},k+h \mathbb{Z},\ldots,k+h \mathbb{Z},x+h \mathbb{Z})$ est une solution de $(E_{h})$. Or, par la proposition \ref{52}, la solution $(k+h \mathbb{Z})$-monomiale minimale de $(E_{h})$ est irréductible. Donc, par le lemme \ref{31}, $x \equiv 0, k [h]$.
\\
\item $5^{m}$ divise $N$. Le $d$-uplet $(x+5^{m} \mathbb{Z},k+5^{m} \mathbb{Z},\ldots,k+5^{m} \mathbb{Z},x+5^{m} \mathbb{Z})$ est une solution de $(E_{5^{m}})$. Or, comme $k$ n'est pas divisible par $5$, la solution $(k+5^{m} \mathbb{Z})$-monomiale minimale de $(E_{5^{m}})$ est irréductible (Théorème \ref{33}). Par le lemme \ref{31}, $x \equiv 0, k [5^{m}]$.
\\
\item $30$ divise $N$. Le $d$-uplet $(x+30 \mathbb{Z},k+30 \mathbb{Z},\ldots,k+30 \mathbb{Z},x+30 \mathbb{Z})$ est une solution de $(E_{30})$. Or, par ce qui précède, la solution $(k+30 \mathbb{Z})$-monomiale minimale de $(E_{30})$ est irréductible. Par le lemme \ref{31}, $x \equiv 0, k [30]$.
\\
\end{itemize}

\noindent Si $x \equiv 0 [h]$ et $x \equiv 0 [5^{m}]$ alors, par le lemme chinois, $\overline{x}=\overline{0}$. Dans ce cas, le $d-2$-uplet $(\overline{k},\ldots,\overline{k})$ est une solution de \eqref{p}, ce qui est absurde (puisque $d-2 < r$). Si $x \equiv k [h]$ et $x \equiv k [5^{m}]$ alors, par le lemme chinois, $\overline{x}=\overline{k}$. Dans ce cas, le $d$-uplet $(\overline{k},\ldots,\overline{k})$ est une solution de \eqref{p}, ce qui est absurde (puisque $d < r$). Donc, on a $x \equiv 0 [h]$ et $x \equiv k [5^{m}]$ ou $x \equiv k [h]$ et $x \equiv 0 [5^{m}]$.
\\
\\Si $x \equiv 0 [h]$ et $x \equiv k [5^{m}]$. Par le lemme chinois, on a $x \equiv ukh [N]$. Si $x \equiv 0 [30]$ alors $5$ divise $ukh$, ce qui est absurde. Donc, $x \equiv k [30]$ et 3 divise $k(uh-1)$. Or, ceci est absurde car $(uh-1)$ et $k$ ne sont pas divisibles par 3.
\\
\\Si $x \equiv k [h]$ et $x \equiv 0 [5^{m}]$. Par le lemme chinois, on a $x \equiv vk5^{m} [N]$. Si $x \equiv 0 [30]$ alors 3 divise $vk5^{m}$, ce qui est absurde. Donc, $x \equiv k [30]$ et $5$ divise $k(v5^{m}-1)$. Or, ceci est absurde car $k$ et $(v5^{m}-1)$ ne sont pas divisibles par $5$.
\\
\\Ainsi, la solution $\overline{k}$-monomiale minimale de \eqref{p} est irréductible et $N$ est semi monomialement irréductible.

\end{proof}

\begin{examples}
{\rm $90=2 \times 3^{2} \times 5$ et $150=2 \times 3 \times 5^{2}$ sont semi monomialement irréductibles.
}
\end{examples}

Cela termine la preuve du théorème \ref{211}. En revanche, contrairement aux théorèmes \ref{26} et \ref{27}, on n'a pas de classification complète. On donne en annexe \ref{D} la liste des entiers semi monomialement irréductibles pairs inférieurs à 2500 et on formule le problème ouvert suivant :

\begin{pro}

Établir une classification complète des entiers semi monomialement irréductibles.

\end{pro}

Notons qu'à la lueur des résultats obtenus dans cette section, il semble que la résolution de ce problème soit étroitement liée à la bonne connaissance des propriétés des solutions monomiales minimales de $(E_{p})$ lorsque $p$ est un nombre premier.
\\
\\
\\
\noindent {\bf Remerciements}.
Je remercie Alain Ninet pour son aide dans l'amélioration des programmes informatiques utilisés pour obtenir certains éléments donnés en annexe.

\appendix

\newpage 

\section{Entiers quasi monomialement irréductibles inférieurs à 1000}
\label{A}

La liste suivante donne les nombres quasi monomialement irréductibles inférieurs à 1000. Les nombres écrits en {\color{blue} bleu} sont premiers. Ceux écrits en {\color{red} rouge} sont de la forme $p^{n}$ avec $p$ premier et $n \geq 2$. Les nombres écrits en {\color{ufogreen} vert} sont ceux de la forme $2^{n}3^{m}$, avec $n, m \geq 1$.
\\
\\ {\color{blue} 2}, {\color{blue} 3}, {\color{red} 4}, {\color{blue} 5}, {\color{ufogreen} 6}, {\color{blue} 7}, {\color{red} 8}, {\color{red} 9}, {\color{blue} 11}, {\color{ufogreen} 12}, {\color{blue} 13}, {\color{red} 16}, {\color{blue} 17}, {\color{ufogreen} 18}, {\color{blue} 19}, {\color{blue} 23}, {\color{ufogreen} 24}, {\color{red} 25}, {\color{red} 27}, {\color{blue} 29}, {\color{blue} 31}, {\color{red} 32}, {\color{ufogreen} 36}, {\color{blue} 37}, {\color{blue} 41}, {\color{blue} 43}, {\color{blue} 47}, {\color{ufogreen} 48}, {\color{red} 49}, {\color{blue} 53}, {\color{ufogreen} 54}, {\color{blue} 59}, {\color{blue} 61}, {\color{red} 64}, {\color{blue} 67}, {\color{blue} 71}, {\color{ufogreen} 72}, {\color{blue} 73}, {\color{blue} 79}, {\color{red} 81}, {\color{blue} 83}, {\color{blue} 89}, {\color{ufogreen} 96}, {\color{blue} 97}, {\color{blue} 101}, {\color{blue} 103}, {\color{blue} 107}, {\color{ufogreen} 108}, {\color{blue} 109}, {\color{blue} 113}, {\color{red} 121}, {\color{red} 125}, {\color{blue} 127}, {\color{red} 128}, {\color{blue} 131}, {\color{blue} 137}, {\color{blue} 139}, {\color{ufogreen} 144}, {\color{blue} 149}, {\color{blue} 151}, {\color{blue} 157}, {\color{ufogreen} 162}, {\color{blue} 163}, {\color{blue} 167}, {\color{red} 169}, {\color{blue} 173}, {\color{blue} 179}, {\color{blue} 181}, {\color{blue} 191}, {\color{ufogreen} 192}, {\color{blue} 193}, {\color{blue} 197}, {\color{blue} 199}, {\color{blue} 211}, {\color{ufogreen} 216}, {\color{blue} 223}, {\color{blue} 227}, {\color{blue} 229}, {\color{blue} 233}, {\color{blue} 239}, {\color{blue} 241}, {\color{red} 243}, {\color{blue} 251}, {\color{red} 256}, {\color{blue} 257}, {\color{blue} 263}, {\color{blue} 269}, {\color{blue} 271}, {\color{blue} 277}, {\color{blue} 281}, {\color{blue} 283}, {\color{ufogreen} 288}, {\color{red} 289}, {\color{blue} 293}, {\color{blue} 307}, {\color{blue} 311}, {\color{blue} 313}, {\color{blue} 317}, {\color{ufogreen} 324}, {\color{blue} 331}, {\color{blue} 337}, {\color{red} 343}, {\color{blue} 347}, {\color{blue} 349}, {\color{blue} 353}, {\color{blue} 359}, {\color{red} 361}, {\color{blue} 367}, {\color{blue} 373}, {\color{blue} 379}, {\color{blue} 383}, {\color{ufogreen} 384}, {\color{blue} 389}, {\color{blue} 397}, {\color{blue} 401}, {\color{blue} 409}, {\color{blue} 419}, {\color{blue} 421}, {\color{blue} 431}, {\color{ufogreen} 432}, {\color{blue} 433}, {\color{blue} 439}, {\color{blue} 443}, {\color{blue} 449}, {\color{blue} 457}, {\color{blue} 461}, {\color{blue} 463}, {\color{blue} 467}, {\color{blue} 479}, {\color{ufogreen} 486}, {\color{blue} 487}, {\color{blue} 491}, {\color{blue} 499}, {\color{blue} 503}, {\color{blue} 509}, {\color{red} 512}, {\color{blue} 521}, {\color{blue} 523}, {\color{red} 529}, {\color{blue} 541}, {\color{blue} 547}, {\color{blue} 557}, {\color{blue} 563}, {\color{blue} 569}, {\color{blue} 571}, {\color{ufogreen} 576}, {\color{blue} 577}, {\color{blue} 587}, {\color{blue} 593}, {\color{blue} 599}, {\color{blue} 601}, {\color{blue} 607}, {\color{blue} 613}, {\color{blue} 617}, {\color{blue} 619}, {\color{red} 625}, {\color{blue} 631}, {\color{blue} 641}, {\color{blue} 643}, {\color{blue} 647}, {\color{ufogreen} 648}, {\color{blue} 653}, {\color{blue} 659}, {\color{blue} 661}, {\color{blue} 673}, {\color{blue} 677}, {\color{blue} 683}, {\color{blue} 691}, {\color{blue} 701}, {\color{blue} 709}, {\color{blue} 719}, {\color{blue} 727}, {\color{red} 729}, {\color{blue} 733}, {\color{blue} 739}, {\color{blue} 743}, {\color{blue} 751}, {\color{blue} 757}, {\color{blue} 761}, {\color{ufogreen} 768}, {\color{blue} 769}, {\color{blue} 773}, {\color{blue} 787}, {\color{blue} 797}, {\color{blue} 809}, {\color{blue} 811}, {\color{blue} 821}, {\color{blue} 823}, {\color{blue} 827}, {\color{blue} 829}, {\color{blue} 839}, {\color{red} 841}, {\color{blue} 853}, {\color{blue} 857}, {\color{blue} 859}, {\color{blue} 863}, {\color{ufogreen} 864}, {\color{blue} 877}, {\color{blue} 881}, {\color{blue} 883}, {\color{blue} 887}, {\color{blue} 907}, {\color{blue} 911}, {\color{blue} 919}, {\color{blue} 929}, {\color{blue} 937}, {\color{blue} 941}, {\color{blue} 947}, {\color{blue} 953}, {\color{red} 961}, {\color{blue} 967}, {\color{blue} 971}, {\color{ufogreen} 972}, {\color{blue} 977}, {\color{blue} 983}, {\color{blue} 991}, {\color{blue} 997}.

\section{Nombre de solutions monomiales minimales irréductibles de $(E_{2^{n}3^{m}})$ lorsque $nm \geq 1$ et $2^{n}3^{m} \leq 1000$}
\label{B}

Si $N$ est un entier, on note $\varphi(N)$ l'indicatrice d'Euler évaluée en $N$ et $\omega(N)$ le nombre de solutions monomiales minimales irréductibles de \eqref{p}. Si $N$ est la puissance d'un nombre premier alors le théorème \ref{33} donne la valeur de $\omega(N)$. Dans les deux tableaux ci-dessous, on a réuni les valeurs de $\omega(N)$ pour $N=2^{n}3^{m} \leq 1000$ (avec $nm \geq 1$), obtenues informatiquement, et on les compare à $\varphi(N)=2^{n}3^{m-1}$.
 \newline

\begin{center}
\begin{tabular}{|c|c|c|c|c|c|c|c|c|c|c|c|c|c|}
\hline
   $N$ & 6 & 12 & 18 & 24 & 36 & 48 & 54 & 72 & 96 & 108 & 144 & 162 & 192  \rule[-7pt]{0pt}{18pt} \\
	\hline
	$\varphi(N)$ & 2 & 4 & 6 & 8 & 12 & 16 & 18 & 24 & 32 & 36 & 48 & 54 & 64   \rule[-7pt]{0pt}{18pt} \\
	\hline
	$\omega(N)$ & 5 & 11 & 15 & 23 & 31 & 41 & 39 & 63 & 77 & 79 & 113 & 111 & 149   \rule[-7pt]{0pt}{18pt} \\
	\hline
	
\end{tabular}
\end{center}
	
\vspace{0.5cm}
	
\begin{center}
\begin{tabular}{|c|c|c|c|c|c|c|c|c|c|c|c|}
\hline
   $N$  & 216 & 288 & 324 & 384 & 432 & 486 & 576 & 648 & 768 & 864 & 972   \rule[-7pt]{0pt}{18pt} \\
	\hline
	$\varphi(N)$  & 72 & 96 & 108 & 128 & 144 & 162 & 192 & 216 & 256 & 288 & 324   \rule[-7pt]{0pt}{18pt} \\
	\hline
	$\omega(N)$  & 159 & 213 & 223 & 293 & 281 & 327 & 413 & 447 & 581 & 525 & 655    \rule[-7pt]{0pt}{18pt} \\
	\hline
	
\end{tabular}
\end{center}	

\section{Liste des solutions monomiales minimales réductibles de $(E_{2^{n}3^{m}})$ pour quelques valeurs de $n$ et de $m$}
\label{C}

On donne ci-dessous, pour chaque $N$ choisi, la liste (obtenue informatiquement) des entiers $0 \leq k \leq N-1$ pour lesquels la solution $\overline{k}$-monomiale minimale de \eqref{p} est réductible :

\begin{itemize}
\item $N=48=2^{4} \times 3$,
\\ (0, 4, 12, 20, 28, 36, 44);
\\
\item $N=108=2^{2} \times 3^{3}$,
\\ (0, 3, 6, 12, 15, 18, 21, 24, 30, 33, 36, 39, 42, 48, 51, 57, 60, 66, 69, 72, 75, 78, 84, 87, 90, 93, 96, 102, 105);
\\
\item $N=192=2^{6} \times 3$,
\\ (0, 4, 8, 12, 16, 20, 24, 28, 36, 40, 44, 48, 52, 56, 60, 68, 72, 76, 80, 84, 88, 92, 100, 104, 108, 112, 116, 120, 124, 132, 136, 140, 144, 148, 152, 156, 164, 168, 172, 176, 180, 184, 188);
\\
\item $N=216=2^{3} \times 3^{3}$,
\\ (0, 3, 6, 12, 15, 18, 21, 24, 30, 33, 36, 39, 42, 48, 51, 57, 60, 66, 69, 72, 75, 78, 84, 87, 90, 93, 96, 102, 105, 111, 114, 120, 123, 126, 129, 132, 138, 141, 144, 147, 150, 156, 159, 165, 168, 174, 177, 180, 183, 186, 192, 195, 198, 201, 204, 210, 213);
\\
\item $N=384=2^{7} \times 3$,
\\ (0, 4, 8, 12, 16, 20, 24, 28, 32, 36, 40, 44, 48, 52, 56, 60, 68, 72, 76, 80, 84, 88, 92, 96, 100, 104, 108, 112, 116, 120, 124, 132, 136, 140, 144, 148, 152, 156, 160, 164, 168, 172, 176, 180, 184, 188, 196, 200, 204, 208, 212, 216, 220, 224, 228, 232, 236, 240, 244, 248, 252, 260, 264, 268, 272, 276, 280, 284, 288, 292, 296, 300, 304, 308, 312, 316, 324, 328, 332, 336, 340, 344, 348, 352, 356, 360, 364, 368, 372, 376, 380);
\\
\item $N=864=2^{5} \times 3^{3}$,
\\ (0, 3, 4, 6, 8, 12, 15, 18, 20, 21, 24, 28, 30, 33, 36, 39, 40, 42, 44, 48, 51, 52, 56, 57, 60, 66, 68, 69, 72, 75, 76, 78, 84, 87, 88, 90, 92, 93, 96, 100, 102, 104, 105, 108, 111, 114, 116, 120, 123, 124, 126, 129, 132,136, 138, 140, 141, 144, 147, 148, 150, 152, 156, 159, 164, 165, 168, 172, 174, 177, 180, 183, 184, 186, 188, 192, 195, 196, 198, 200, 201, 204, 210, 212, 213, 216, 219, 220, 222, 228, 231, 232, 234, 236, 237, 240, 244, 246, 248, 249, 252, 255, 258, 260, 264, 267, 268, 273, 276, 280, 282, 284, 285, 288, 291, 292, 294, 296, 300, 303, 306, 308, 309, 312, 316, 318, 321, 324, 327, 328, 330, 332, 336, 339, 340, 342, 344, 345, 348, 354, 356, 357, 360, 363, 364, 366, 372, 375, 376, 380, 381, 384, 388, 390, 392, 393, 396, 399, 402, 404, 408, 411, 412, 414, 417, 420, 424, 426, 428, 429, 435, 436, 438, 440, 444, 447, 450, 452, 453, 456, 460, 462, 465, 468, 471, 472, 474, 476, 480, 483, 484, 488, 489, 492, 498, 500, 501, 504, 507, 508, 510, 516, 519, 520, 522, 524, 525, 528, 532, 534, 536, 537, 540, 543, 546, 548, 552, 555, 556, 558, 561, 564, 568, 570, 572, 573, 576, 579, 580, 582, 584, 588, 591, 596, 597, 600, 604, 606, 609, 612, 615, 616, 618, 620, 624, 627, 628, 630, 632, 633, 636, 642, 644, 645, 648, 651, 652, 654, 660, 663, 664, 666, 668, 669, 672, 676, 678, 680, 681, 684, 687, 690, 692, 696, 699, 700, 705, 708, 712, 714, 716, 717, 720, 723, 724, 726, 728, 732, 735, 738, 740, 741, 744, 748, 750, 753, 756, 759, 760, 762, 764, 768, 771, 772, 774, 776, 777, 780, 786, 788, 789, 792, 795, 796, 798, 804, 807, 808, 812, 813, 816, 820, 822, 824, 825, 828, 831, 834, 836, 840, 843, 844, 846, 849, 852, 856, 858, 860, 861).

\end{itemize}

\section{Entiers semi monomialement irréductibles pairs inférieurs à 2500}
\label{D}

La liste suivante donne la liste (obtenue informatiquement) des nombres semi monomialement irréductibles pairs inférieurs à 2500. Les nombres écrits en {\color{blue} bleu} sont de la forme $2p^{n}$ avec $p$ premier et $n \geq 1$. Ceux écrits en {\color{ufogreen} vert} sont ceux donnés dans le corollaire \ref{56} et la proposition \ref{57} (à l'exception des nombres de la forme $2^{n}$, $2\times 3^{n}$ et $2 \times 5^{n}$ déjà écrits en bleu). Les nombres écrits en {\color{red} rouge} sont ceux sur lequel on ne possède pas de résultat théorique établissant leur caractère semi monomialement irréductible.
\\
\\{\color{blue} 4}, {\color{blue} 6}, {\color{blue} 8}, {\color{blue} 10}, {\color{ufogreen} 12}, {\color{blue} 14}, {\color{blue} 16}, {\color{blue} 18}, {\color{ufogreen} 20}, {\color{blue} 22}, {\color{ufogreen} 24}, {\color{blue} 26}, {\color{ufogreen} 28}, {\color{ufogreen} 30}, {\color{blue} 32}, {\color{blue} 34}, {\color{ufogreen} 36}, {\color{blue} 38}, {\color{ufogreen} 40}, {\color{blue} 46}, {\color{ufogreen} 48}, {\color{blue} 50}, {\color{blue} 54}, {\color{ufogreen} 56}, {\color{blue} 58}, {\color{ufogreen} 60}, {\color{blue} 62}, {\color{blue} 64}, {\color{red} 66}, {\color{ufogreen} 68}, {\color{ufogreen} 72}, {\color{blue} 74}, {\color{red} 78}, {\color{ufogreen} 80}, {\color{blue} 82}, {\color{ufogreen} 84}, {\color{blue} 86}, {\color{ufogreen} 90}, {\color{blue} 94}, {\color{ufogreen} 96}, {\color{blue} 98}, {\color{ufogreen} 100}, {\color{blue} 106}, {\color{ufogreen} 108}, {\color{ufogreen} 112}, {\color{blue} 118}, {\color{ufogreen} 120}, {\color{blue} 122}, {\color{ufogreen} 124}, {\color{blue} 128}, {\color{red} 132}, {\color{blue} 134}, {\color{ufogreen} 136}, {\color{ufogreen} 140}, {\color{blue} 142}, {\color{ufogreen} 144}, {\color{blue} 146}, {\color{ufogreen} 150}, {\color{red} 156}, {\color{blue} 158}, {\color{ufogreen} 160}, {\color{blue} 162}, {\color{blue} 166}, {\color{ufogreen} 168}, {\color{blue} 178}, {\color{ufogreen} 180}, {\color{ufogreen} 192}, {\color{blue} 194}, {\color{ufogreen} 196}, {\color{red} 198}, {\color{ufogreen} 200}, {\color{blue} 202}, {\color{ufogreen} 204}, {\color{blue} 206}, {\color{blue} 214}, {\color{ufogreen} 216}, {\color{blue} 218}, {\color{red} 222}, {\color{ufogreen} 224}, {\color{blue} 226}, {\color{red} 234}, {\color{ufogreen} 240}, {\color{blue} 242}, {\color{ufogreen} 248}, {\color{blue} 250}, {\color{ufogreen} 252}, {\color{blue} 254}, {\color{blue} 256}, {\color{blue} 262}, {\color{red} 264}, {\color{ufogreen} 270}, {\color{ufogreen} 272}, {\color{blue} 274}, {\color{red} 276}, {\color{blue} 278}, {\color{ufogreen} 280}, {\color{ufogreen} 288}, {\color{blue} 298}, {\color{ufogreen} 300}, {\color{blue} 302}, {\color{red} 312}, {\color{blue} 314}, {\color{ufogreen} 320}, {\color{ufogreen} 324}, {\color{blue} 326}, {\color{red} 330}, {\color{blue} 334}, {\color{ufogreen} 336}, {\color{blue} 338}, {\color{ufogreen} 340}, {\color{blue} 346}, {\color{blue} 358}, {\color{ufogreen} 360}, {\color{blue} 362}, {\color{ufogreen} 372}, {\color{blue} 382}, {\color{ufogreen} 384}, {\color{blue} 386}, {\color{red} 390}, {\color{ufogreen} 392}, {\color{blue} 394}, {\color{red} 396}, {\color{blue} 398}, {\color{ufogreen} 400}, {\color{ufogreen} 408}, {\color{ufogreen} 420}, {\color{blue} 422}, {\color{ufogreen} 432}, {\color{red} 444}, {\color{blue} 446}, {\color{ufogreen} 448}, {\color{ufogreen} 450}, {\color{blue} 454}, {\color{blue} 458}, {\color{blue} 466}, {\color{red} 468}, {\color{ufogreen} 476}, {\color{blue} 478}, {\color{ufogreen} 480}, {\color{blue} 482}, {\color{blue} 486}, {\color{ufogreen} 496}, {\color{ufogreen} 500}, {\color{blue} 502}, {\color{ufogreen} 504}, {\color{ufogreen} 508}, {\color{blue} 512}, {\color{blue} 514}, {\color{blue} 526}, {\color{red} 528}, {\color{blue} 538}, {\color{ufogreen} 540}, {\color{blue} 542}, {\color{ufogreen} 544}, {\color{red} 552}, {\color{blue} 554}, {\color{ufogreen} 560}, {\color{blue} 562}, {\color{red} 564}, {\color{blue} 566}, {\color{ufogreen} 576}, {\color{blue} 578}, {\color{blue} 586}, {\color{ufogreen} 588}, {\color{red} 594}, {\color{ufogreen} 600}, {\color{ufogreen} 612}, {\color{blue} 614}, {\color{ufogreen} 620}, {\color{blue} 622}, {\color{red} 624}, {\color{blue} 626}, {\color{blue} 634}, {\color{ufogreen} 640}, {\color{red} 642}, {\color{ufogreen} 648}, {\color{red} 654}, {\color{red} 660}, {\color{blue} 662}, {\color{red} 666}, {\color{ufogreen} 672}, {\color{blue} 674}, {\color{ufogreen} 680}, {\color{blue} 686}, {\color{blue} 694}, {\color{blue} 698}, {\color{ufogreen} 700}, {\color{red} 702}, {\color{blue} 706}, {\color{blue} 718}, {\color{ufogreen} 720}, {\color{blue} 722}, {\color{red} 726}, {\color{blue} 734}, {\color{ufogreen} 744}, {\color{blue} 746}, {\color{ufogreen} 750}, {\color{ufogreen} 756}, {\color{blue} 758}, {\color{blue} 766}, {\color{ufogreen} 768}, {\color{blue} 778}, {\color{red} 780}, {\color{ufogreen} 784}, {\color{red} 792}, {\color{blue} 794}, {\color{ufogreen} 800}, {\color{blue} 802}, {\color{ufogreen} 810}, {\color{ufogreen} 816}, {\color{blue} 818}, {\color{red} 828}, {\color{blue} 838}, {\color{ufogreen} 840}, {\color{blue} 842}, {\color{red} 852}, {\color{red} 858}, {\color{blue} 862}, {\color{ufogreen} 864}, {\color{blue} 866}, {\color{ufogreen} 868}, {\color{red} 876}, {\color{blue} 878}, {\color{blue} 886}, {\color{red} 888}, {\color{ufogreen} 896}, {\color{blue} 898}, {\color{ufogreen} 900}, {\color{blue} 914}, {\color{blue} 922}, {\color{red} 924}, {\color{blue} 926}, {\color{blue} 934}, {\color{red} 936}, {\color{ufogreen} 952}, {\color{blue} 958}, {\color{ufogreen} 960}, {\color{ufogreen} 972}, {\color{blue} 974}, {\color{ufogreen} 980}, {\color{blue} 982}, {\color{red} 990}, {\color{ufogreen} 992}, {\color{blue} 998}
\\
\\{\color{ufogreen} 1000}, {\color{blue} 1006}, {\color{ufogreen} 1008}, {\color{red} 1014}, {\color{ufogreen} 1016}, {\color{blue} 1018}, {\color{ufogreen} 1020}, {\color{blue} 1024}, {\color{red} 1028}, {\color{blue} 1042}, {\color{blue} 1046}, {\color{red} 1056}, {\color{blue} 1058}, {\color{ufogreen} 1080}, {\color{blue} 1082}, {\color{ufogreen} 1088}, {\color{red} 1092}, {\color{blue} 1094}, {\color{red} 1104}, {\color{red} 1110}, {\color{blue} 1114}, {\color{ufogreen} 1116}, {\color{ufogreen} 1120}, {\color{blue} 1126}, {\color{red} 1128}, {\color{blue} 1138}, {\color{blue} 1142}, {\color{ufogreen} 1152}, {\color{blue} 1154}, {\color{ufogreen} 1156}, {\color{red} 1164}, {\color{red} 1170}, {\color{blue} 1174}, {\color{ufogreen} 1176}, {\color{blue} 1186}, {\color{red} 1188}, {\color{blue} 1198}, {\color{ufogreen} 1200}, {\color{blue} 1202}, {\color{blue} 1214}, {\color{ufogreen} 1224}, {\color{blue} 1226}, {\color{blue} 1234}, {\color{blue} 1238}, {\color{ufogreen} 1240}, {\color{red} 1248}, {\color{blue} 1250}, {\color{ufogreen} 1260}, {\color{blue} 1262}, {\color{ufogreen} 1280}, {\color{blue} 1282}, {\color{red} 1284}, {\color{blue} 1286}, {\color{blue} 1294}, {\color{ufogreen} 1296}, {\color{blue} 1306}, {\color{red} 1308}, {\color{blue} 1318}, {\color{red} 1320}, {\color{blue} 1322}, {\color{red} 1332}, {\color{ufogreen} 1344}, {\color{blue} 1346}, {\color{ufogreen} 1350}, {\color{blue} 1354}, {\color{ufogreen} 1360}, {\color{blue} 1366}, {\color{ufogreen} 1372}, {\color{red} 1380}, {\color{blue} 1382}, {\color{ufogreen} 1400}, {\color{blue} 1402}, {\color{red} 1404}, {\color{blue} 1418}, {\color{ufogreen} 1428}, {\color{blue} 1438}, {\color{ufogreen} 1440}, {\color{red} 1452}, {\color{blue} 1454}, {\color{blue} 1458}, {\color{blue} 1466}, {\color{blue} 1478}, {\color{blue} 1486}, {\color{ufogreen} 1488}, {\color{ufogreen} 1500}, {\color{blue} 1502}, {\color{ufogreen} 1512}, {\color{blue} 1514}, {\color{blue} 1522}, {\color{ufogreen} 1524}, {\color{ufogreen} 1536}, {\color{blue} 1538}, {\color{blue} 1546}, {\color{red} 1560}, {\color{ufogreen} 1568}, {\color{blue} 1574}, {\color{red} 1584}, {\color{blue} 1594}, {\color{ufogreen} 1600}, {\color{blue} 1618}, {\color{ufogreen} 1620}, {\color{blue} 1622}, {\color{ufogreen} 1632}, {\color{blue} 1642}, {\color{blue} 1646}, {\color{red} 1650}, {\color{blue} 1654}, {\color{red} 1656}, {\color{blue} 1658}, {\color{blue} 1678}, {\color{ufogreen} 1680}, {\color{blue} 1682}, {\color{red} 1692}, {\color{ufogreen} 1700}, {\color{red} 1704}, {\color{blue} 1706}, {\color{blue} 1714}, {\color{red} 1716}, {\color{blue} 1718}, {\color{blue} 1726}, {\color{ufogreen} 1728}, {\color{ufogreen} 1736}, {\color{red} 1752}, {\color{blue} 1754}, {\color{blue} 1762}, {\color{ufogreen} 1764}, {\color{blue} 1766}, {\color{blue} 1774}, {\color{red} 1776}, {\color{red} 1782}, {\color{ufogreen} 1792}, {\color{ufogreen} 1800}, {\color{blue} 1814}, {\color{blue} 1822}, {\color{ufogreen} 1836}, {\color{blue} 1838}, {\color{red} 1848}, {\color{blue} 1858}, {\color{ufogreen} 1860}, {\color{red} 1872}, {\color{blue} 1874}, {\color{blue} 1882}, {\color{blue} 1894}, {\color{ufogreen} 1904}, {\color{blue} 1906}, {\color{ufogreen} 1920}, {\color{blue} 1922}, {\color{red} 1926}, {\color{red} 1932}, {\color{blue} 1934}, {\color{blue} 1942}, {\color{ufogreen} 1944}, {\color{red} 1950}, {\color{blue} 1954}, {\color{ufogreen} 1960}, {\color{red} 1962}, {\color{blue} 1966}, {\color{red} 1980}, {\color{blue} 1982}, {\color{ufogreen} 1984}, {\color{blue} 1994}, {\color{red} 1998}
\\
\\{\color{ufogreen} 2000}, {\color{ufogreen} 2016}, {\color{blue} 2018}, {\color{blue} 2026}, {\color{red} 2028}, {\color{ufogreen} 2032}, {\color{blue} 2038}, {\color{ufogreen} 2040}, {\color{blue} 2042}, {\color{blue} 2048}, {\color{red} 2056}, {\color{blue} 2062}, {\color{blue} 2066}, {\color{blue} 2078}, {\color{blue} 2098}, {\color{ufogreen} 2100}, {\color{blue} 2102}, {\color{red} 2106}, {\color{ufogreen} 2108}, {\color{red} 2112}, {\color{blue} 2122}, {\color{blue} 2126}, {\color{blue} 2138}, {\color{ufogreen} 2160}, {\color{blue} 2174}, {\color{ufogreen} 2176}, {\color{red} 2178}, {\color{blue} 2182}, {\color{red} 2184}, {\color{blue} 2186}, {\color{blue} 2194}, {\color{blue} 2206}, {\color{red} 2208}, {\color{blue} 2218}, {\color{red} 2220}, {\color{ufogreen} 2232}, {\color{blue} 2234}, {\color{ufogreen} 2240}, {\color{red} 2244}, {\color{blue} 2246}, {\color{ufogreen} 2250}, {\color{red} 2256}, {\color{blue} 2258}, {\color{ufogreen} 2268}, {\color{red} 2292}, {\color{blue} 2302}, {\color{ufogreen} 2304}, {\color{blue} 2306}, {\color{ufogreen} 2312}, {\color{red} 2316}, {\color{blue} 2326}, {\color{red} 2328}, {\color{red} 2340}, {\color{blue} 2342}, {\color{ufogreen} 2352}, {\color{blue} 2362}, {\color{blue} 2374}, {\color{red} 2376}, {\color{ufogreen} 2380}, {\color{blue} 2386}, {\color{ufogreen} 2400}, {\color{blue} 2402}, {\color{blue} 2426}, {\color{ufogreen} 2430}, {\color{blue} 2434}, {\color{red} 2442}, {\color{blue} 2446}, {\color{ufogreen} 2448}, {\color{blue} 2458}, {\color{blue} 2462}, {\color{blue} 2474}, {\color{ufogreen} 2480}, {\color{red} 2484}, {\color{red} 2496}, {\color{blue} 2498}, {\color{ufogreen} 2500}.
\\
\\À partir des éléments ci-dessus, on peut, en utilisant le théorème \ref{54}, obtenir de nouveaux nombres semi monomialement irréductibles. Par exemple, $132=2^{2} \times 33$ est semi moniomialement irréductible et $4 \times 31$ l'est aussi, par le corollaire \ref{56}. Donc, par le théorème \ref{54}, $4092 =2^{2} \times 33 \times 31$ est semi moniomialement irréductible.


\begin{thebibliography}{99}

\bibitem{C} 
M.~Cuntz,
{\it A combinatorial model for tame frieze patterns}, 
Münster J. Math., Vol 12 no. 1, (2019), pp 49-56.

\bibitem{CH} 
M. Cuntz, T. Holm,
{\it Frieze patterns over integers and other subsets of the complex numbers}, 
J. Comb. Algebra., Vol. 3 no. 2, (2019), pp 153-188.

\bibitem{M}
F. Mabilat,
\textit{Quelques éléments de combinatoire des matrices de $SL(2,\mathbb{Z})$}, 
Bulletin des Sciences Mathématiques, Vol. 167, Article 102958, (2021), https://doi.org/10.1016/j.bulsci.2021.102958.

\bibitem{M1}
F. Mabilat,
{\it Combinatoire des sous-groupes de congruence du groupe modulaire},
Annales Mathématiques Blaise Pascal, Vol. 28 no. 1, (2021), pp. 7-43. doi : 10.5802/ambp.398. https://ambp.centre-mersenne.org/articles/10.5802/ambp.398/. 

\bibitem{M2}
F. Mabilat,
{\it Combinatoire des sous-groupes de congruence du groupe modulaire II}. 
Annales Mathématiques Blaise Pascal, Vol. 28 no. 2, (2021), pp. 199-229. doi : 10.5802/ambp.404. https://ambp.centre-mersenne.org/articles/10.5802/ambp.404/.

\bibitem{M3}
F. Mabilat,
\textit{Entiers monomialement irréductibles}, hal-03487145, arXiv:2112.10410.

\bibitem{M4}
F. Mabilat,
\textit{Solutions monomiales minimales irréductibles dans $SL_{2}(\mathbb{Z}/p^{n}\mathbb{Z})$}, Bulletin des Sciences Mathématiques, Vol. 194, Article 103456, (2024), https://doi.org/10.1016/j.bulsci.2024.103456.

\bibitem{OEIS}
The OEIS Foundation Inc.
\textit{The On-Line Encyclopedia of Integer Sequences} [en ligne]. Disponible à l'adresse : https://oeis.org [consulté le 12 mai 2023].

\bibitem{O} 
V. Ovsienko, 
{\it Partitions of unity in $SL(2,\mathbb{Z})$,  negative continued fractions,  and dissections of polygons,} 
Res. Math. Sci., Vol. 5 no. 2, (2018), Article 21, 25 pp. 

\bibitem{Ri} 
H. Riesel. 
{\it Prime numbers and computer methods for factorization}. 
Springer Science $\&$ Business Media, 2012.

\bibitem{WZ}
M. Weber, M. Zhao,
{\it Factorization of frieze patterns,}
Revista de la Unión Matemática Argentina, Vol. 60 no. 2, (2019), pp 407-415.

\end{thebibliography}
\end{document}